\documentclass{article}

\usepackage{arxiv}

\usepackage[utf8]{inputenc} 
\usepackage[T1]{fontenc}    
\usepackage{hyperref}     
\usepackage{url}           
\usepackage{booktabs}     
\usepackage{amsfonts}     
\usepackage{amsthm}
\usepackage{nicefrac}    
\usepackage{microtype}  
\usepackage{lipsum}		 
\usepackage{graphicx}
\usepackage{natbib}
\usepackage{doi}
\usepackage{kpfonts}
\usepackage{multirow}

\usepackage{float}
\usepackage{caption}

\usepackage{xcolor}

\newtheorem{theorem}{Theorem}[section]

\newtheorem{lemma}[theorem]{Lemma}
\newtheorem{corollary}[theorem]{Corollary}

\newtheorem{remark}[theorem]{Remark}

\newtheorem{assumption}[theorem]{Assumption}
\newtheorem{problem}[theorem]{Problem}

\newcommand{\R}{\mathbb{R}}
\newcommand{\N}{\mathbb{N}}
\newcommand{\DD}{\mathcal{D}}
\newcommand{\Dw}{\mathcal{D}(\omega)}
\newcommand{\DDw}{\overline{\mathcal{D}(\omega)}}
\newcommand{\Dr}{\mathcal{D}_{\text{ref}}}
\newcommand{\DDr}{\overline{\mathcal{D}_{\text{ref}}}}
\newcommand{\Ar}{A_{\text{ref}}}
\newcommand{\Fr}{f_{\text{ref}}}
\newcommand{\Ur}{u_{\text{ref}}}

\newcommand{\T}{\mathcal{T}}
\newcommand{\dd}{d\lambda^d(x) \, dP(\omega)}
\newcommand{\ddm}{d\lambda^d(x) \, d\mu(y)}

\newcommand{\LH}{L^2(\Omega;H_0^1(\Dr;\R))}

\newcommand{\tLH}{L^{2;(M)}(\Gamma;H_0^1(\Dr;\R))}

\newcommand{\tB}{\bar{B}^{(M)}}
\newcommand{\tF}{\bar{F}^{(M)}}

\newcommand{\tE}{\bar{\mathcal{E}}^{(M)}}
\newcommand{\tu}{\bar{u}^{(M)}}
\newcommand{\tur}{\bar{u}^{(M)}_{\text{ref}}}
\newcommand{\tar}{\bar{A}^{(M)}_{\text{ref}}}
\newcommand{\tfr}{\bar{f}^{(M)}_{\text{ref}}}

\title{Deep learning methods for stochastic Galerkin approximations of random domain problems.}
\date{} 					

\author{{\hspace{1mm}Fabio Musco} \\
	Department of Mathematics\\
	University of Stuttgart\\
	\texttt{fabio.musco@ians.uni-stuttgart.de} \\
	\And
	{\hspace{1mm}Andrea Barth} \\
	Department of Mathematics\\
	University of Stuttgart\\
	\texttt{andrea.barth@ians.uni-stuttgart.de} \\
}

\renewcommand{\headeright}{}
\renewcommand{\undertitle}{}
\renewcommand{\shorttitle}{Deep learning methods for random domain problems}

\begin{document}
\bibliographystyle{plain}
\setcitestyle{numbers}
\maketitle

\begin{abstract}
This work considers strong and weak stochastic Galerkin approximations of random domain problems for the elliptic random partial differential equation (PDE). The random domain problems are realized by sufficiently regular random domain mapping, allowing a transformation on deterministic reference domain.
A traditional numerical method for solving the resulting high-dimensional coupled stochastic Galerkin systems is replaced by deep learning techniques. 
We compare a physics-informed neural network approach, based on the strong stochastic Galerkin residual, with a Deep Ritz approach, based on the weak stochastic Galerkin system, reformulated as Ritz energy minimization.
The neural networks serve as surrogates for the deterministic spectral coefficients of the respective stochastic Galerkin solution. The stochastic parameters are not used as neural network input parameters and the stochastic domain is not sampled during training. The efficiency of the methods is demonstrated on a randomly stretched interval and a randomly deformed annulus in two spatial dimensions. 
\end{abstract}



\section{Introduction}

The elliptic partial differential equation (PDE) is a fundamental model for diffusion processes, such as heat conduction, electrostatics or subsurface flows. In many applications, however, the underlying setting is not known exactly and, as a consequence, the uncertainty enters the mathematical model, resulting in an elliptic random partial differential equation. The elliptic random partial differential equation serves as a central benchmark problem in the field of uncertainty quantification, particularly as a prominent model for subsurface flows. In this work, we consider uncertainties in the domain of the diffusion itself, namely random domain problems for the elliptic PDE. 

A common approach to random domain problems is to describe admissible geometries through sufficiently regular random deformations of a given deterministic reference domain.
These random deformations are realized by a given random field, the so-called random domain mapping, which allows the governing equation to be transformed to an elliptic random PDE posed on this deterministic reference domain (see e.g.\ \cite{CastrillonCandas2016, CastrillonCandas2021, Hakula2024, Harbrecht2016, Harbrecht2021}). 
Under this transformation, the domain uncertainty is transferred to the driving functions of the PDE, resulting in a stochastic diffusion coefficient and stochastic forcing term.

In recent decades, a variety of techniques have been developed for quantifying the uncertainty in the resulting system. These methods can be broadly classified into two categories, intrusive and non-intrusive methods. We refer to \cite{Field2015, Tuminaro2011} for an overview and comparison of both approaches. 
The non-intrusive methods encompass all sampling based solution techniques, such as Monte Carlo methods, stochastic collocation methods and their respective variants (see e.g.\ \cite{Babuska2007, Barth2011, Cliffe2011, Graham2011, Kuo2016, Nobile2008}).
In this work, we follow a different objective. Rather than estimating statistics of the stochastic solution by repeatedly obtaining different deterministic solutions, we determine a spectral approximation of the solution random field itself. This is achieved using the intrusive stochastic Galerkin method, based on polynomial chaos expansions of the involved random fields (see e.g.\ \cite{Babuska2004, Babuska2005, Cohen2010, Deb2001, Ernst2011, Ghanem1999, Ghanem1991, Karniadakis2005, LeMaitre2010, Mugler2011, Mugler2013, Schwab2011, Todor2007, Wiener1938, Xiu2010, Xiu2002, XiuWienerAskey}). Stochastic Galerkin methods combined with the random domain mapping approach have been studied, for example by Xiu and Tartakovsky \cite{Xiu2006} and Kundu et al. \cite{Kundu2014}, primarily in the context of weak formulations. In this work, the existing strong stochastic Galerkin framework is extended to the setting of random domain mappings.

The polynomial chaos expansion describes a spectral expansion of the random fields with respect to a global orthonormal tensor product polynomial basis in the stochastic variables. 
For the purpose of the stochastic Galerkin method, polynomial chaos expansions of the involved random fields are obtained and truncated at a maximal total polynomial degree, resulting in finite spectral representations.
The global stochastic Galerkin method substitutes these finite spectral representations into the PDE with a subsequent Galerkin projection over the stochastic domain. The resulting stochastic Galerkin system is a usually high dimensional system of coupled deterministic PDEs for the solution spectral coefficients. 
A variety of techniques may be employed in order to solve the resulting coupled system, according to classical approaches. The high dimensional coupling of the equations, however, poses difficulties for practical applications with classical numerical solution techniques. Consequently, the maximal dimension of the coupled system is usually bounded by a relatively small number or specific decoupling strategies or sparse tensor methods are employed in the literature for special cases (see e.g.\ \cite{Ernst2011, Karniadakis2005, Mugler2011, Mugler2013, Schwab2011, Xiu2009}).
A common decoupling method is based on a specific double orthogonal basis, resulting in an exponentially growing number of decoupled partial differential equations (see \cite{Babuska2004, Babuska2005, Frauenfelder2005}). 
This decoupling approach, however, is only applicable in special cases, namely if the driving random fields are multilinear combinations of the driving random variables, which is not satisfied for general transformed random domain problems.

The central idea of this work is to combine the intrusive stochastic Galerkin method with a deep learning based solution approach. 
There are many ways to approximate solutions of partial differential equations by deep learning approaches, we refer to \cite{Beck2023, Blechschmidt2021, Cuomo2022, Karniadakis2021} for an overview. In particular, deep neural networks have been successfully applied to solve high dimensional partial differential equations, see e.g.\ \cite{Beck2019, Beck2023, Han2018, Musco2024, Nabian2019, Sirignano2018}.
The deep learning approaches developed in this work are based on the strong and weak stochastic Galerkin formulations. Deep neural networks are used as surrogates for the deterministic spectral coefficients of the polynomial chaos expansion of the corresponding stochastic Galerkin solutions.

The neural network training strategies based on the strong stochastic Galerkin formulations are derived by minimizing a least-squares type objective function constructed from the stochastic Galerkin projection of the strong residual of the transformed PDE. 
This strong residual based deep learning approach is closely related to the well-known physics-informed neural networks, popularized by Raissi et al.\ \cite{Raissi2019}, inspired by the work of Lagaris et al.\ \cite{Lagaris1997, Lagaris1998}. 
In contrast to a classical physics-informed neural network setting, our proposed training objective is not based directly on the strong residual of a parameterized partial differential equation, but on the projected residual of the corresponding intrusive strong stochastic Galerkin system. 
Tho combination of a physics-informed neural network solution method and random domain problems have previously been studied, see e.g.\ \cite{Bharadwaja2022, Burbulla2023, Kashefi2023, Kashefi2022}.
Bharadwaja et al.\ \cite{Bharadwaja2022} consider geometric uncertainty described by a single scalar parameter that is used as a neural network input parameter during training.
By contrast, in our proposed approach, the neural networks serve as surrogates for the deterministic spectral coefficients of the polynomial chaos representation of the stochastic Galerkin solution and therefore depend only on the spatial variable.
Consequently, the stochastic parameters do not enter the neural network as input variables or need to be sampled during the training procedure.
In \cite{Kashefi2023, Kashefi2022}, the authors propose to train neural networks simultaneously on different geometries. These collections of different geometries are, however, not formulated as random domains.
Burbulla \cite{Burbulla2023} extends the deterministic physics-informed neural network approach to smoothly deformed geometries. This approach as well as the transformation, however, is purely deterministic, and no polynomial chaos expansion or stochastic Galerkin projection is employed.
To the best of our knowledge, the combination of a random domain mapping and a physics-informed neural network approach, based on the intrusive stochastic Galerkin solution of a corresponding transformed reference formulation of the random domain problem has not previously been studied.
In particular, the random domain framework accommodates general random domain deformations represented by sufficiently regular random domain mappings.

The second neural network training strategy is based on the weak stochastic Galerkin formulation derived from the stochastic weak formulation of the transformed elliptic random PDE.
For deterministic PDEs, there are several works that extend the classical physics-informed neural network methods to deep learning approaches based on the weak formulation of PDEs. These approaches are commonly referred to as variational physics-informed neural networks, and related variants, see e.g.\ \cite{Berrone2022, Kharazmi2019, Kharazmi2021, KhodayiMehr2020}. In these methods, the optimization scheme is based on the weak residual, where the variational equation has to be evaluated against a prescribed collection of test functions in every training step. The corresponding integrals are approximated by numerical quadrature rules during the optimization process. 
Such approaches, however, are computationally very expensive and scale poorly with increasing dimensions, number of elements, test functions, and quadrature points limiting its use in real world applications and complex geometries, see e.g.\ \cite{Abda2025, Anandh2025, Liu2023}. 
While there are methods to computationally improve this, we follow a different approach, called the Deep Ritz Method, introduced by E and Yu \cite{E2018} for deterministic PDEs. In this approach, a neural network is trained to minimize the Ritz energy functional in order to approximate the weak solution of the PDE. 
This approach avoids the selection of a family of test functions and the evaluation of the weak residual for every test function in every training step, while still benefiting from the reduced order of differentiation and lower regularity assumptions that come with the weak formulation. It is, however, only valid for specific PDEs, where the weak problem admits an equivalent energy minimization problem. 
Versions of the Deep Ritz method have also been studied in the context of PDEs on varying domains. 
Kaltenbach and Zeinhofer \cite{Kaltenbach2025} work with parameterized varying geometries. The parameter describing the domain geometry is used as an input parameter of the neural network. Their training strategy is based on the Ritz energy functional that is additionally integrated over the parameter space of the varying domains. During training, the parameter space of the domain uncertainty as well as the spatial domain of the PDE are sampled or discretized.
In \cite{Nguyen2025}, the authors propose a geometry aware deep learning framework, based on energy minimization, extending the framework to more general geometries. The training procedure is separated into two stages, where geometry information are processed in the first stage by a different autoencoder or variational autoencoder neural network. In addition to the spatial coordinates, the neural network to approximate the PDE solution takes either parameters of a parameterized random geometry or the output of the neural network, used in the first stage, as input variables.
Analogously to the strong residual based training strategy, the neural networks in our proposed method are used as surrogates for the stochastic Galerkin spectral coefficients.
Hence, the stochastic parameters do not enter the neural network as input variables or need to be sampled during training. This developed approach provides a random domain framework, suitable for uncertainty quantification, allowing for a broad class of random domains, represented by sufficiently regular random domain mappings.
To the best of our knowledge, the combination of a random domain mapping and a deep learning approach, based on the Ritz energy minimization of the stochastic Galerkin system of the transformed problem, has not been addressed before in the literature. 

In Section~\ref{Section:PDE_on_RD}, we formulate the random domain problem for the elliptic PDE and the corresponding transformed elliptic random PDE on the deterministic reference domain that serves as basis for the subsequent stochastic Galerkin formulations. These are derived in Section~\ref{Section:SpectralExpansion}. For the stochastic weak formulation, the corresponding Ritz energy minimization problem is established. Furthermore, we assemble the corresponding stochastic Galerkin systems by explicitly carrying out the stochastic projections to obtain the computational forms of the respective stochastic Galerkin formulations. The computational forms provide the implementation basis of the deep learning training strategies derived in Section~\ref{sec:DL}. This constitutes the methodological contribution of this work. We derive the computational architectures of the neural networks called \textit{S-GalerkinNet}, based on the strong stochastic Galerkin residual, and the \textit{S-RitzNet}, based on the weak Ritz energy minimization. 
In Section~\ref{Section:Numerical_Experiments}, we present numerical experiments for the developed deep learning approaches. We consider two random domain problems, a randomly stretched interval in one spatial dimension, and a randomly deformed annulus in two spatial dimensions. The experiments compare the strong residual based and the weak Ritz energy based approaches with regards to approximation quality, computational complexity, training times and optimization behavior.



\section{Elliptic partial differential equation on random domains} \label{Section:PDE_on_RD}
Let $(\Omega, \mathcal{A}, P)$ be a complete and countably generated probability space and let $\Dw \subset \R^d,$ for any $\omega \in \Omega$, $d \in \N$, be an open, bounded and connected random domain with Lipschitz boundary $\partial \Dw$ and closure $\overline{\Dw}$. We consider the elliptic partial differential equation (PDE) with homogeneous Dirichlet boundary conditions on the random domain (RD). For the general setup of the random domain problem, we follow \citep{Hakula2024, Harbrecht2016}.

\begin{problem}[Strong RD formulation of the elliptic PDE] \label{Problem:StrongRD_Formulation}
Find the solution $u: \Omega \times \overline{\Dw} \to \R$, such that
\begin{align*}
- \nabla \cdot (a(x) \nabla u(\omega, x)) 
	& = 
f(x) 
	&& \hspace*{-3cm}
x \in \Dw, \, \omega \in \Omega, 	
	\\
u(\omega, x) 
	& =
0
	&& \hspace*{-3cm}
x \in \partial \Dw, \, \omega \in \Omega,
\end{align*}
where the arising differential operators are understood with respect to the spatial variable $x \in \Dw$.
\end{problem}
We consider the diffusion coefficient $a:\mathbf{D} \to \R$, modeling the permeability, and the forcing term $f: \mathbf{D} \to \R$, representing a forcing or source, to be defined on an open hold-all domain $\mathbf{D} \subset \R^d$, such that
\begin{align*}
\Dw \subset\mathbf{D} \text{ for every } \omega \in \Omega.
\end{align*}
We recall the definitions of the involved function spaces (see also \citep[Ch. 2]{Musco2024}). Let $
\DD \subset \R^d$, $p \in (0, \infty]$, and let
\begin{align*}
L^p(\DD; \R) := L^p((\DD , \mathcal{B}(\DD), \lambda^d); (\R , \mathcal{B}(\R)))
	= \left\lbrace \phi : \DD \to \R \text{ measurable} \, : \, \lVert \phi \rVert_{L^p(\DD; \R)} < \infty \right\rbrace, 
\end{align*}
denote the Lebesgue space, where $\mathcal{B}(\DD)$ denotes the Borel $\sigma$-algebra on $\DD$ and $\lambda^d$ the Lebesgue-measure. The corresponding (quasi-)norms are given by
\begin{align*}
\lVert \phi \rVert_{L^p(\DD; \R)} = \left( \int_\DD \lvert \phi(x) \rvert^p \, d\lambda^d(x) \right)^{1/p},
\end{align*} 
for $p \in (0, \infty)$, and
\begin{align*}
\lVert \phi \rVert_{L^{\infty}(\DD; \R)} = \inf \{c \geq 0 \, : \, \lvert \phi(x) \rvert \leq c \text{ for } \lambda^d \text{-a.e.\ } x \in \DD\},
\end{align*}
for $p = \infty$. Note that $\lVert \cdot \rVert_{L^p(\DD;\R)}$ defines a norm for $p \geq 1$, and a quasi-norm for $p \in (0,1)$. 
Furthermore, let $k \in \N$ and let
\begin{align*}
H^k(\DD; \R) := H^k((\DD, \mathcal{B}(\DD), \lambda^d); (\R, \mathcal{B}(\R)))
= \{ \phi \in L^2(\DD; \R) \, : \, D^\alpha \phi \in L^2(\DD; \R) \text{ for every } \alpha \in \N_0^d \text{ with } \lvert \alpha \rvert \leq k\},
\end{align*}
denote the $k$-th order Sobolev space, where $\alpha = (\alpha_1, \dots , \alpha_d) \in \N_0^d$ denotes a multiindex of order $\lvert \alpha \rvert = \alpha_1 + \dots + \alpha_d$, and 
\begin{align*}
D^\alpha \phi = \frac{\partial^{\alpha_1}}{\partial x_1^{\alpha_1}} \cdots \frac{\partial^{\alpha_d}}{\partial x_d^{\alpha_d}} \phi
\end{align*}
denotes the $\alpha$-th weak partial derivative of $\phi$. The defined Sobolev space, equipped with the norm
\begin{align*}
\lVert \phi \rVert_{H^k(\DD; \R) } = 
\bigg( \sum_{\lvert \alpha \rvert \leq k} \int_\DD \lvert D^\alpha \phi \rvert^2 \, d\lambda^d(x) \bigg)^{1/2},
\end{align*}
and corresponding inner product, is a Hilbert space. In the following, we denote the weak derivatives by the usual derivative notation.
We further introduce the first-order Sobolev space of functions that vanish at the boundary in the sense of traces via
\begin{align*}
H_0^1(\mathcal{D}; \R) := H_0^1((\mathcal{D}, \mathcal{B}(\mathcal{D}), \lambda^d) ; (\R, \mathcal{B}(\R)))
	= \left\lbrace \phi \in H^1(\DD; \R) \, : \,  \phi\big|_{\partial \DD} = 0 \text{ in the sense of traces}\right\rbrace,
\end{align*}
equipped with the (semi-) norm
\begin{align*}
\lVert \phi \rVert_{H_0^1(\DD; \R)} = \left( \int_\DD \lVert \nabla \phi(x) \rVert_2^2 \, d\lambda^d(x) \right)^{1/2},
\end{align*}
where $\lVert \cdot \rVert_2$ denotes the Euclidean norm on $\mathbb{R}^d$. Note that this defines a norm on $H_0^1(\DD; \R)$ due to the Poincaré inequality. The space $H_0^1(\DD; \R)$, equipped with the inner product
\begin{align*}
\langle \phi, \psi \rangle_{H_0^1(\DD; \R)} = \int_\DD \langle \nabla \phi(x), \nabla \psi(x) \rangle_2 \, d\lambda^d(x),
\end{align*}
is a separable Hilbert space. 
To extend the notion of real-valued random variables to those taking values in function spaces, we introduce the Lebesgue-Bochner spaces.
Let $(E, \lVert \cdot \rVert_E)$ be a Banach space and let $\mathcal{B}(E)$ denote the Borel $\sigma$-algebra on $E$, that is the smallest $\sigma$-algebra generated by the $\lVert \cdot \rVert_E$-open sets. 
We call a mapping $X: \Omega \to E$ \textit{strongly measurable random field}, if it is $\mathcal{A}-\mathcal{B}(E)$-measurable and has a separable range, i.e. $X(\Omega) \subset E$ is $\lVert \cdot \rVert_E$-separable. 
Let $p \in [1, \infty]$, and let
\begin{align*}
L^p(\Omega; E) := L^p((\Omega, \mathcal{A}, P); (E, \mathcal{B}(E)))
	= \left\lbrace \phi: \Omega \to E \text{ strongly measurable} \, : \, \lVert \phi \rVert_{L^p(\Omega; E)} < \infty \right\rbrace,
\end{align*}
denote the Lebesgue-Bochner space with the norm
\begin{align*}
\lVert \phi \rVert_{L^p(\Omega; E)} 
	&=
\left( \int_\Omega \lVert \phi(\omega) \rVert_E^p \, dP(\omega) \right)^{1/p}, \text{ for } p \in [1, \infty), 
	\\
\lVert \phi \rVert_{L^{\infty}(\Omega; E)} 
	&=
\inf \{c \geq 0 \, : \, \lVert \phi(\omega) \rVert_E \leq c \text{ for } P-\text{a.e. } \omega \in \Omega \}. 
\end{align*}
If $E$ is a separable Hilbert space, then the Lebesgue--Bochner space $L^2(\Omega; E)$ is a separable Hilbert space that is isomorphic to the tensor product Hilbert space $L^2(\Omega; E) \cong L^2(\Omega; \R) \otimes E$ (see \citep[Remark 2.19]{Schwab2011}).
We refer to \citep[App. E]{Cohn2013} for the derivation and more details on the Lebesgue--Bochner spaces and integrals. 
We denote the space of real-valued rank $k \in \N, \, k \geq 2$, tensors of dimensions $\varprod_{j = 1}^k d$ by $\mathcal{M}_k^d$. The space $\mathcal{M}_k^d$, equipped with the Frobenius inner product
\begin{align*}
\langle A, B \rangle_{F_k} := \sum_{i_1, \dots , i_k = 1}^d A_{i_1, \dots , i_k} B_{i_1, \dots , i_k},
\end{align*}
defines a separable Hilbert space $(\mathcal{M}_k^d, \langle \cdot , \cdot \rangle_{F_k})$ with induced norm
\begin{align*}
\lVert A \rVert_{F_k} = \left( \sum_{i_1, \dots , i_k = 1}^d A_{i_1, \dots , i_k}^2 \right)^{1/2}.
\end{align*}
Note that for $k = 2$, the Frobenius norm of a matrix $A \in \mathcal{M}_2^d$ has the representation
\begin{align*}
\lVert A \rVert_{F_2} = \sqrt{\sigma_1 + \dots + \sigma_d},
\end{align*}
where $\sqrt{\sigma_1}, \dots , \sqrt{\sigma_d} \geq 0$ denote the singular values of $A \in \mathcal{M}_2^d$, i.e.\ $\sigma_1, \dots , \sigma_d \geq 0$ denote the eigenvalues of $A^T A \in \mathcal{M}_2^d$.
Furthermore, for $k \in \N$, let
\begin{align*}
C^k(\DD; \R^d) := \left\lbrace \phi: \DD \to \R^d \, k\text{-times continuously differentiable} \, : \, \lVert \phi \rVert_{C^k(\DD; \R^d)} < \infty \right\rbrace
\end{align*}
denote the space of continuous functions that are $k$-times differentiable with continuous derivatives on $\DD$, equipped with the norm
\begin{align*}
\lVert \phi \rVert_{C^k(\DD; \R^d)} 
= \lVert \phi \rVert_{C^0(\DD ; \R^d)} + \sum_{j = 1}^k \lVert D^j \phi \rVert_{C^0(\DD; \mathcal{M}_{j+1}^d)}
:= \sup_{x \in \DD} \lVert \phi(x) \rVert_2 + \sum_{j = 1}^k \sup_{x \in \DD} \lVert D^j \phi(x) \rVert_{F_{j+1}}, 
\end{align*}
where $D^j \phi: \DD \to \mathcal{M}_{j+1}^d$ denotes the $j$-th derivative of $\phi$. In the following, we denote the Jacobian $D^1 \phi$ of a vector field $\phi$ by $J_\phi : \DD \to \mathcal{M}_2^d$.
We impose the following usual assumptions on the diffusion coefficient and forcing term.
\begin{assumption} \label{Assumption:Bounds_Det_Diffusivity}
Assume $a \in L^{\infty}(\mathbf{D}; \R)$, and that there exist constants $a_-, a_+ \geq 0$, such that
\begin{align*}
0 < a_- \leq a(x) \leq a_+ < \infty,
\end{align*}
for every $x \in \mathbf{D}$. Furthermore, let the forcing term $f: \mathbf{D} \to \R$ be a member of $L^2(\mathbf{D} ; \R)$.
\end{assumption}
In the following we transform the random domain problem to a reference problem on a deterministic domain. For that, we define an open bounded and connected deterministic reference domain $\Dr \subset \R^d$ with Lipschitz boundary $\partial \Dr$ and closure $\DDr$. We further impose the following assumptions on the random domains $\Dw, \omega \in \Omega$.
\begin{assumption} \label{Assumption:RD_And_Trafo_C1}
Assume there exists a random field $\T: \Omega \times \DDr \to \R^d$, the so-called random domain mapping, satisfying additional assumptions stated below, such that for $\omega \in \Omega$
\begin{align*}
\DDw = \T(\omega, \DDr).
\end{align*}
Furthermore the random domain mapping $\T(\omega, \cdot): \DDr \to \DDw$ is a $C^1$-diffeomorphism for every $\omega \in \Omega$, i.e.\ $\T(\omega, \cdot) \in C^1(\DDr; \DDw)$ is bijective and $\T^{-1}(\omega, \cdot) \in C^1(\DDw; \DDr)$ for every $\omega \in \Omega$. Additionally, we assume that there exists a uniform bound $C > 0$, such that
\begin{align*}
\lVert \T(\omega, \cdot) \rVert_{C^1(\DDr; \R^d)} \leq C, \quad \lVert \T^{-1}(\omega, \cdot) \rVert_{C^1(\DDw; \R^d)} \leq C
\end{align*}
for every $\omega \in \Omega$.
\end{assumption}
Under Assumption~\ref{Assumption:Bounds_Det_Diffusivity} and standard arguments, the pathwise bilinear form $\overline{B}: H_0^1(\Dw; \R) \times H_0^1(\Dw, \R) \to \R$ for $\omega \in \Omega$, defined by
\begin{align*}
\overline{B}(u,v;\omega) = \int_{\Dw} a(x) \langle \nabla u(x) , \nabla v(x) \rangle_2 \, d\lambda^d(x),
\end{align*}
is continuous and coercive.
This in turn guarantees, by the Lax--Milgram Lemma (see e.g.\ \citep[Ch. 6.2, Thm. 1]{Evans2022}), the existence and uniqueness of a solution to the pathwise weak formulation of Problem \ref{Problem:StrongRD_Formulation}, stated as follows, for any $\omega \in \Omega$.
\begin{problem}[Pathwise weak RD formulation of the elliptic PDE] \label{Problem:PathwiseWeakFormulation_RD_Problem}
For a given forcing term $f \in L^2(\mathbf{D}; \R)$ and $\omega \in \Omega$, find $u(\omega, \cdot) \in H_0^1(\Dw ; \R)$, such that
\begin{align*}
\overline{B}(u(\omega, \cdot) ,v;\omega) = \overline{F}(v; \omega) \text{ for every } v \in H_0^1(\Dw ; \R),
\end{align*}
where $\overline{F}(v; \omega) := \int_{\Dw} f(x) v(x) \, d\lambda^d(x)$.
\end{problem}
By a transformation, we derive an equivalent formulation on the deterministic reference domain. For $\omega \in \Omega$ 
\begin{align*}
\int_{\Dw} a(x) \langle \nabla u(\omega, x) , \nabla v(x) \rangle_2 \, d\lambda^d(x)
	&=
\int_{\T(\omega, \Dr)} a(x)\langle \nabla u(\omega, x) , \nabla v(x) \rangle_2 \, d\lambda^d(x) 
	\\
	&=
\int_{\Dr} a(\T(\omega, x)) \langle (\nabla u) (\omega, \T(\omega, x)) , (\nabla v) (\T(\omega, x)) \rangle_2 \cdot \lvert \operatorname{det}(J_{\T}(\omega, x)) \rvert \, d\lambda^d(x).
\end{align*}
By the chain rule, we get
\begin{align*}
\nabla u(\omega, \T (\omega, x)) = J_\T^T(\omega, x) (\nabla u) (\omega, \T (\omega, x)),
\end{align*}
where $J_\T$ denotes the Jacobian of $\T$. Let us define the transformed functions
\begin{align*}
\widehat{a}(\omega, x) := a(\T (\omega,x)), \, \widehat{u}(\omega, x) := u(\omega, \T (\omega, x)), \, \widehat{v}(\omega, x) := v(\T (\omega, x)).
\end{align*}
Therefore, we continue the equation above as
\begin{align*}
\int_{\Dw} a(x) \langle \nabla u(\omega, x) , \nabla v(x) \rangle_2 \, d\lambda^d(x)
	&=
\int_{\Dr} \langle \lvert \operatorname{det}(J_{\T}(\omega, x)) \rvert \widehat{a}(\omega, x) J_\T^{-1}(\omega, x) J_\T^{-T}(\omega, x) \nabla \widehat{u}(\omega, x) , \nabla \widehat{v}(\omega, x) \rangle_2 \, d\lambda^d(x),
\end{align*}
with $J_\T^{-T}(\omega, x) := (J_\T^T(\omega, x))^{-1}$.
Analogously, with $\widehat{f}(\omega, x) := f(\T(\omega, x))$, we get for the right-hand side
\begin{align*}
\int_{\Dw} f(x) v(x) \, dx = \int_{\Dr} \lvert \operatorname{det}(J_{\T}(\omega, x)) \rvert \widehat{f}(\omega, x) \widehat{v}(\omega, x) \, dx.
\end{align*}
Finally, we denote
\begin{align*}
\Ar(\omega, x) := \lvert \operatorname{det}(J_{\T}(\omega, x)) \rvert \widehat{a}(\omega, x) J_\T^{-1}(\omega, x) J_\T^{-T}(\omega, x) , 
\quad
\Fr(\omega, x) := \lvert \operatorname{det}(J_{\T}(\omega, x)) \rvert \widehat{f}(\omega, x).
\end{align*}
\begin{theorem} \label{Thm:EstimateSmallestAndLargestSingularValue}
Let $x \in \DDr, \omega \in \Omega$ and let $\sigma_{\min}(\omega, x), \sigma_{\max}(\omega, x) \in \R$ denote the smallest and largest eigenvalue of $J_\T^{-1}(\omega, x) J_\T^{-T}(\omega, x) \in \mathcal{M}_2^d$. Under Assumption~\ref{Assumption:RD_And_Trafo_C1}, there exist constants $0 < \sigma_- \leq \sigma_+ < \infty$, such that for every $\omega \in \Omega$ and $x \in \DDr$
\begin{align*}
0 < \sigma_- \leq \sigma_{\min}(\omega, x) \leq \sigma_{\max}(\omega, x) \leq \sigma_+ < \infty.
\end{align*}
\end{theorem}
\begin{proof}
The matrix $J_\T^{-1}(\omega, x) J_\T^{-T}(\omega, x)$ is symmetric and positive definite for every $x \in \DDr, \omega \in \Omega$, and hence the ordered eigenvalues of $J_\T^{-1}(\omega, x) J_\T^{-T}(\omega, x)$  satisfy $0 < \sigma_1(\omega, x) \leq \dots \leq \sigma_d(\omega, x) < \infty$ with $\sigma_{\min}(\omega, x) = \sigma_1(\omega,x)$ and $\sigma_{\max}(\omega, x) = \sigma_d(\omega, x)$.
By the inverse function theorem, we get $J_{\T^{-1}}(\omega, x) = J_\T^{-1}(\omega, x)$, and with Assumption~\ref{Assumption:RD_And_Trafo_C1}
\begin{align*}
\lVert J_{\T^{-1}}(\omega, \cdot) \rVert_{C^0(\DDw; \mathcal{M}_2^d)}
	&=
\max_{x \in \DDr} \lVert J_\T^{-1}(\omega, x) \rVert_{F_2}
= \max_{x \in \DDr} \lVert J_\T^{-T}(\omega, x) \rVert_{F_2}
= \max_{x \in \DDr} \sqrt{\sigma_{1}(\omega, x) + \dots + \sigma_d(\omega, x)}
	\\
	& \leq
\lVert \T^{-1}(\omega, \cdot) \rVert_{C^1(\DDw ; \R^d)} \leq C,
\end{align*}
for every $\omega \in \Omega$. 
Analogously, we get
\begin{align*}
\lVert J_{\T}(\omega, \cdot) \rVert_{C^0(\DDr; \mathcal{M}_2^d)}
= \max_{x \in \DDr} \lVert J_\T(\omega, x) \rVert_{F_2}
= \max_{x \in \DDr} \sqrt{\frac{1}{\sigma_1(\omega, x)} + \dots + \frac{1}{\sigma_d(\omega, x)}} \leq \lVert \T(\omega, \cdot) \rVert_{C^1(\DDr ; \R^d)} \leq C,
\end{align*}
for every $\omega \in \Omega$, since the eigenvalues of $J_\T^T(\omega, x)J_\T(\omega, x)$ are the reciprocals of the eigenvalues of $J_\T^{-1}(\omega, x)J_\T^{-T}(\omega, x) = (J_\T^T(\omega, x)J_\T(\omega, x))^{-1}$. Therefore, there exist constants $\sigma_-, \sigma_+ > 0$ such that
\begin{align*}
0 < \sigma_- \leq \sigma_{\min}(\omega, x) \leq \sigma_{\max}(\omega, x) \leq \sigma_+ < \infty
\end{align*}
for every $x \in \DDr, \omega \in \Omega$.
\end{proof}
\begin{corollary} \label{Cor:EstimateFunctionalDeterminant}
Under Assumption~\ref{Assumption:RD_And_Trafo_C1}, there exist constants $0 < \tau_- \leq \tau_+ < \infty$, such that for every $\omega \in \Omega$ and $x \in \DDr$ 
\begin{align*}
0 < \tau_- \leq \lvert \operatorname{det} J_\T(\omega, x) \rvert \leq \tau_+ < \infty.
\end{align*}
\end{corollary}
\begin{proof}
The absolute value of the determinant of a matrix is the product of its singular values. Therefore, let $\sqrt{\sigma_1(\omega, x)}, \dots , \sqrt{\sigma_d(\omega, x)}$ denote the singular values of $J_\T^{-T}(\omega, x)$, then we conclude for $x \in \DDr, \omega \in \Omega$
\begin{align*}
\lvert \operatorname{det} J_\T(\omega, x) \rvert = \frac{1}{\lvert \operatorname{det} J_\T^{-T}(\omega, x) \rvert} = \prod_{i = 1}^d \frac{1}{\sqrt{\sigma_i(\omega, x)}}.
\end{align*}
Finally, by Theorem~\ref{Thm:EstimateSmallestAndLargestSingularValue} we get for every $x \in \DDr, \omega \in \Omega$
\begin{align*}
0 < \sigma_+^{-d/2} \leq \prod_{i = 1}^d \frac{1}{\sqrt{\sigma_{\max}(\omega, x)}} 
\leq \lvert \operatorname{det} J_\T(\omega, x) \rvert
\leq \prod_{i = 1}^d \frac{1}{\sqrt{\sigma_{\min}(\omega, x)}} \leq \sigma_-^{-d/2} < \infty.
\end{align*}
The statement follows with $\tau_- := \sigma_+^{-d/2}, \tau_+ := \sigma_-^{-d/2}$.
\end{proof}
As a direct consequence of Assumption~\ref{Assumption:Bounds_Det_Diffusivity} and Corollary~\ref{Cor:EstimateFunctionalDeterminant}, we get $\Fr(\omega, \cdot) \in L^2(\Dr; \R)$ for every $\omega \in \Omega$. With that, we formulate the transformed problem on the deterministic reference domain.
\begin{problem}[Pathwise weak formulation of the transformed elliptic PDE] \label{Problem:PathwiseWeakTransformedProblem}
For $\omega \in \Omega$, and the transformed forcing term $\Fr (\omega, \cdot) \in  L^2(\Dr; \R)$ find $\widehat{u}(\omega, \cdot) \in H_0^1(\Dr;\R)$, such that 
\begin{align*}
\int_{\Dr} \langle \Ar (\omega, x) \nabla \widehat{u}(\omega, x) , \nabla \widehat{v}(\omega, x) \rangle_2 \, d\lambda^d(x)
= \int_{\Dr} \Fr (\omega, x) \widehat{v}(\omega, x) \, d\lambda^d(x)
\end{align*}
for all transformed $\widehat{v}(\omega, \cdot) \in H_0^1(\Dr ; \R)$.
\end{problem}
Through the transformation, we shifted the stochasticity from the domain into the driving functions of the PDE. The solution space $H_0^1(\Dw;\R)$ of the random domain problem is transported to the solution space $H_0^1(\Dr;\R)$ of the transformed problem. As a result of Corollary~\ref{Cor:EstimateFunctionalDeterminant}, these solution spaces are isomorphic via the isomorphism $\iota : H_0^1(\Dr ; \R) \to H_0^1(\Dw;\R)$ (see also \citep[Lm. 1]{Harbrecht2016}), such that for $\omega \in \Omega, x \in \Dw$ and $u \in  H_0^1(\Dr; \R)$
\begin{align*}
\iota(u)(x) = u(\T^{-1}(\omega, x)).
\end{align*}
Therefore, instead of testing against the transformed functions $\widehat{v}(\omega,\cdot) = v(\T(\omega, \cdot)) \in H_0^1(\Dr;\R)$ in Problem~\ref{Problem:PathwiseWeakTransformedProblem}, we can test against $v \in H_0^1(\Dr;\R)$ independent of $\omega \in \Omega$.
\begin{remark}
Let $\omega \in \Omega$, and let $u(\omega, \cdot) \in H_0^1(\Dw;\R)$ be a solution of the weak random domain problem, Problem~\ref{Problem:PathwiseWeakFormulation_RD_Problem}, and let $\widehat{u}(\omega, \cdot) \in H_0^1(\Dr;\R)$ be a solution of the weak transformed problem, Problem~\ref{Problem:PathwiseWeakTransformedProblem}, then for every $x \in \Dr$ and $\hat{x} \in \Dw$ 
\begin{align*}
\widehat{u}(\omega, x) = u(\omega, \T(\omega, x)), \quad u(\omega, \hat{x}) = \widehat{u}(\omega, \T^{-1}(\omega, \hat{x})).
\end{align*}
\end{remark}
Problem \ref{Problem:PathwiseWeakTransformedProblem} is the weak formulation of an elliptic random PDE on the deterministic reference domain stated as follows.
\begin{problem}[Strong formulation of the transformed elliptic random PDE] \label{Problem:StrongFormTransformedERPDE}
Find the solution $\Ur : \Omega \times \DDr \to \R$, such that
\begin{align*}
- \nabla \cdot (\Ar(\omega, x) \nabla \Ur (\omega, x))
&= \Fr(\omega, x)
&& \hspace*{-2cm}
x \in \Dr, \, \omega \in \Omega,
\\
\Ur(\omega,x) &= 0
&& \hspace*{-2cm}
x \in \partial \Dr, \, \omega \in \Omega.
\end{align*}
\end{problem}
In order to formulate the stochastic weak formulation of the elliptic random PDE, we define the bilinear form $B: \LH \times \LH \to \R$, and the linear form $F: \LH \to \R$, via
\begin{align*}
B(u,v) 
	& :=
\int_\Omega \int_{\Dr} \langle \Ar(\omega, x) \nabla u(\omega, x) , \nabla v(\omega,x) \rangle_2 \, d\lambda^d(x) \, dP(\omega),
	\\
F(v) & :=
\int_{\Omega} \int_{\Dr} \Fr(\omega, x) v(\omega,x) \, d\lambda^d(x) \, dP(\omega).
\end{align*}
Additionally, we note that by Corollary~\ref{Cor:EstimateFunctionalDeterminant} and Assumption~\ref{Assumption:Bounds_Det_Diffusivity}, we get $\Fr \in L^2(\Omega; L^2(\Dr;\R))$, since for $P$-almost every $\omega \in \Omega$
\begin{align*}
\lVert \Fr \rVert_{L^2(\Omega; L^2(\Dr;\R))}^2 
&=
\int_\Omega \int_{\Dr} \lvert \operatorname{det}(J_{\T}(\omega, x)) \rvert^2 \lvert f(\T(\omega, x)) \rvert^2 \, d\lambda^d(x) \, dP(\omega) 
	\\
	& \leq
\tau_+ \int_\Omega \int_{\Dr} \lvert \operatorname{det}(J_{\T}(\omega, x)) \rvert \, \lvert f(\T(\omega, x)) \rvert^2 \, d\lambda^d(x) \, dP(\omega) 
	\\
	& \leq
\tau_+ \int_\Omega \bigg( \int_{\Dw} \lvert f(x) \rvert^2 \, d\lambda^d(x) \bigg) dP(\omega) 
\leq
\tau_+ \int_\Omega \int_{\mathbf{D}} \lvert f(x) \rvert^2 \, d\lambda^d(x) \, dP(\omega) 
	\\
	&=
\tau_+ \int_{\mathbf{D}} \lvert f(x) \rvert^2 \, d\lambda^d(x) < \infty,
\end{align*}
by the change of variable formula and the hold-all domain $\Dw \subset \mathbf{D}$ for every $\omega \in \Omega$.
The corresponding weak formulation is stated as follows.
\begin{problem}[Stochastic weak formulation of the transformed elliptic random PDE] \label{Problem:StochasticWeakForm_TransformedERPDE}
For the transformed forcing term $\Fr \in L^2(\Omega; L^2(\Dr;\R))$ find $\Ur \in \LH$, such that
\begin{align*}
B(\Ur,v) = F(v) \text{ for every } v \in \LH.
\end{align*}
\end{problem}
We refer to \citep{Babuska2004, Babuska2005} for more details on the stochastic formulation of the elliptic random PDE.
\begin{lemma} \label{Lemma:CourantFischer_Estimate_Frobenius_Aref}
Under Assumption~\ref{Assumption:Bounds_Det_Diffusivity} and Assumption~\ref{Assumption:RD_And_Trafo_C1}, there exist constants $0<\underline{\sigma} \leq \overline{\sigma} < \infty$, such that for every $\xi \in  \R^d, \omega \in \Omega$ and $x \in \DDr$
\begin{align*}
\underline{\sigma} \lVert \xi \rVert_2^2 \leq \langle \xi, \Ar(\omega,x) \xi \rangle_2 \leq \overline{\sigma} \lVert \xi \rVert_2^2.
\end{align*}
Furthermore, $\lVert \Ar(\omega, x) \rVert_{F_2} \leq \sqrt{d} \overline{\sigma}$ for every $\omega \in \Omega, x \in \DDr$.
\end{lemma}
\begin{proof}
By Corollary~\ref{Cor:EstimateFunctionalDeterminant} and Assumption~\ref{Assumption:Bounds_Det_Diffusivity}, we get for every $x \in \DDr, \omega \in \Omega$
\begin{align*}
\lvert \operatorname{det} J_\T(\omega, x) \rvert \geq \tau_- > 0, \quad
\widehat{a}(\omega, x) \geq a_- > 0.
\end{align*}
Therefore $\Ar(\omega, x) = \lvert \operatorname{det} J_\T(\omega, x) \rvert \widehat{a}(\omega, x) J_\T^{-1}(\omega, x) J_\T^{-T} (\omega, x)$ is symmetric, positive definite for every $x \in \DDr, \omega \in \Omega$. Let $\lambda_{\min}(\omega,x), \, \lambda_{\max}(\omega, x) > 0$ denote the smallest and largest eigenvalue of $\Ar(\omega, x)$. Furthermore, let $\sigma_{\min}(\omega, x), \, \sigma_{\max}(\omega, x) > 0$ denote the smallest and largest eigenvalue of $J_\T^{-1}(\omega, x) J_\T^{-T}(\omega, x)$. By Theorem~\ref{Thm:EstimateSmallestAndLargestSingularValue}, we get for every $\omega \in \Omega, x \in \DDr$
\begin{align*}
\lambda_{\min}(\omega, x) 
	&=
\lvert \operatorname{det} J_\T(\omega, x) \rvert \widehat{a}(\omega, x) \sigma_{\min}(\omega, x) 
\geq \tau_- a_- \sigma_- =: \underline{\sigma} > 0,
	\\
\lambda_{\max}(\omega, x)
	&=
\lvert \operatorname{det} J_\T(\omega, x) \rvert \widehat{a}(\omega, x) \sigma_{\max}(\omega, x) 
\leq \tau_+ a_+ \sigma_+ =: \overline{\sigma} < \infty.
\end{align*}
By the Courant-Fischer min-max principle (see e.g.\ \citep[Thm. 4.2.6]{Horn2012}), for every $\xi \in \R^d \setminus \{0\}$ it holds
\begin{align*}
\lambda_{\min}(\omega, x) \leq \frac{\langle \xi, \Ar(\omega, x) \xi \rangle_2}{\langle \xi ,\xi \rangle_2} \leq \lambda_{\max}(\omega, x),
\end{align*}
for every $\omega \in \Omega, x \in \DDr$, which shows the first statement. Since $\Ar(\omega, x)$ is symmetric and positive definite, we get
\begin{align*}
\lVert \Ar(\omega, x) \rVert_{F_2} = \sqrt{\lambda_1^2(\omega, x) + \dots + \lambda_d^2(\omega,x)} \leq \sqrt{d} \lambda_{\max}(\omega, x) \leq \sqrt{d}  \overline{\sigma},
\end{align*}
for every $\omega \in \Omega, x \in \DDr$, where $\lambda_1(\omega, x), \dots, \lambda_d(\omega, x) > 0$ denote the eigenvalues of $\Ar(\omega, x)$. 
\end{proof}
\begin{theorem}
Under Assumption~\ref{Assumption:Bounds_Det_Diffusivity} and Assumption~\ref{Assumption:RD_And_Trafo_C1}, there exists a unique solution $\Ur \in \LH$ of the stochastic weak formulation of the transformed elliptic random PDE, Problem~\ref{Problem:StochasticWeakForm_TransformedERPDE}. 
\end{theorem}
\begin{proof}
For the existence and uniqueness, we show continuity and coercivity of the bilinear form $B:\LH \times \LH \to \R$. For that, we use Lemma~\ref{Lemma:CourantFischer_Estimate_Frobenius_Aref}, as well as that the Frobenius norm is consistent with the Euclidean vector norm, i.e.\ for every $A \in \mathcal{M}_2^d, \xi \in \R^d$, it holds $\lVert A\xi \rVert_2 \leq \lVert A \rVert_{F_2} \lVert \xi \rVert_2$.
\begin{align*}
\lvert B(u,v) \rvert 
	& \leq 
\int_\Omega \int_{\Dr} 
\lvert \langle \Ar(\omega, x) \nabla u(\omega, x), \nabla v(\omega, x) \rangle_2 \rvert \, \dd 
	\\
	& \leq
\int_\Omega \int_{\Dr} 
\lVert \Ar(\omega, x) \rVert_{F_2} \lVert \nabla u(\omega, x) \rVert_2 \lVert \nabla v(\omega, x) \rVert_2 \, \dd 
	\\
	& \leq
\sqrt{d} \overline{\sigma} \int_\Omega \int_{\Dr} 
\lVert \nabla u(\omega, x) \rVert_2 \lVert \nabla v(\omega, x) \rVert_2 \, \dd 
	\\
	& \leq
\sqrt{d} \overline{\sigma} \left(
\int_\Omega \int_{\Dr} 
\lVert \nabla u(\omega, x) \rVert_2^2 \, \dd \right)^{1/2}
\left(
\int_\Omega \int_{\Dr} 
\lVert \nabla v(\omega, x) \rVert_2^2 \, \dd \right)^{1/2}
 \\
 &=
\sqrt{d} \overline{\sigma} \lVert u \rVert_{\LH} \lVert v \rVert_{\LH},
	\\
B(u,u) 
	& =
\int_\Omega \int_{\Dr} 
 \langle \Ar(\omega, x) \nabla u(\omega, x), \nabla u(\omega, x) \rangle_2 \, \dd 
	\\	
	& = 
\int_\Omega \int_{\Dr} 
 \langle  \nabla u(\omega, x), \Ar(\omega, x) \nabla u(\omega, x) \rangle_2 \, \dd 
	\\ 	
 	& \geq
\int_\Omega \int_{\Dr} \underline{\sigma} \lVert \nabla u(\omega, x) \rVert_2^2 \, \dd
= \underline{\sigma} \lVert u \rVert_{\LH}^2.
\end{align*}
Finally we establish the continuity of the linear form $F: \LH \to \R$
\begin{align*}
\lvert F(v) \rvert 
	& \leq 
\int_\Omega \int_{\Dr} \lvert \Fr(\omega, x) v(\omega, x) \rvert \, \dd
	\\	
	& \leq 
\int_\Omega 
\left( \int_{\Dr} \lvert \Fr(\omega, x) \rvert^2 \, d\lambda^d(x) \right)^{1/2}
\left( \int_{\Dr} \lvert v(\omega, x) \rvert^2 \, d\lambda^d(x) \right)^{1/2}
\, dP(\omega)
	\\
	& \leq
C_P \lVert \Fr \rVert_{L^2(\Omega; L^2(\Dr; \R))} \lVert v \rVert_{\LH},
\end{align*}
where $C_P > 0$ denotes the Poincaré constant. By the Lax-Milgram Lemma we get the existence and uniqueness of the solution $\Ur \in \LH$.
\end{proof}

\section{Spectral expansion methods} \label{Section:SpectralExpansion}

\subsection{Polynomial chaos expansion}
We briefly recall the notation for polynomial chaos expansions used throughout this work, and refer to 
\citep{Karniadakis2005, Xiu2002, XiuWienerAskey, Ernst2011, Musco2024} for a detailed discussion. 
Let $(Y_n)_{n = 1, \dots , N}$, for $N \in \N$, denote independent random variables, where $Y_n : \Omega \to \Gamma_n$ is defined on the complete and countably generated probability space $(\Omega, \mathcal{A}, P)$ and image measurable space $(\Gamma_n ,\Sigma_n)$ for Borel subsets $\Gamma_n \subset \R$. For $n = 1, \dots , N$ let $\mu_n := P \circ Y_n^{-1}: \Sigma_n \to [0,1]$ denote the law of $Y_n$.
Furthermore, let $\Gamma := \varprod_{n = 1}^N \Gamma_n$ denote the product space, endowed with the corresponding product $\sigma$-algebra $\Sigma := \bigotimes_{n= 1}^N \Sigma_n$. Then the $\mathcal{A}-\Sigma$-measurable map $Y: \Omega \to \Gamma, \, \omega \mapsto (Y_1(\omega),  \dots , Y_N(\omega))$ has the law $\mu = \bigotimes_{n = 1}^N \mu_n$.
Under suitable assumptions (see e.g.\ \citep{Ernst2011}), there exists an orthonormal polynomial basis, denoted by $(p_{n,i})_{i \in \N_0}$, for $n = 1, \dots, N$ of $L^2(\Gamma_n; \R)$ with respect to the measure $\mu_n$. 
Define the set of index sequences $\mathcal{I}_N := \{ \nu = (\nu_1, \dots , \nu_N)  \in \mathbb{N}_0^N \, : \, \nu_i \in \mathbb{N}_0 \}$, and for every index sequence $\nu \in \mathcal{I}_N$ a corresponding tensor product function
\begin{align*}
p_\nu := \bigotimes_{n=1}^N p_{n, \nu_n}: \Gamma \to \mathbb{R}, \quad y=(y_1, \dots , y_N) \mapsto \prod_{n = 1}^N p_{n, \nu_n}(y_n).
\end{align*}
For every $\nu \in \mathcal{I}_N$, the tensor product function $p_\nu$ is a multivariate polynomial of total degree 
$\lvert \nu \rvert := \nu_1 + \dots + \nu_N$. The tensor product polynomials $(p_\nu)_{\nu \in \mathcal{I}_N}$ form an orthonormal polynomial basis of $L^2(\Gamma; \R)$. We enumerate this basis by the graded lexicographic ordering and denote it by $(p_k)_{k \in \N_0}$. For a maximal total polynomial degree $P \in \N_0$, we consider the truncated index set
\begin{align*}
\mathcal{I}_{N, P} = \{ \nu \in \mathcal{I}_N \, : \, \lvert \nu \rvert = \nu_1 + \dots + \nu_N \leq P \}.
\end{align*}
The corresponding truncated polynomial chaos space has the dimension
\begin{align*}
M + 1 = \lvert \mathcal{I}_{N, P} \rvert = \binom{N + P}{P}.
\end{align*}
The polynomial chaos expansion of $\R$-valued random variables can be extended to certain random fields (see also \citep[Rem.\ 3.12]{Ernst2011}). Let $(H, \langle \cdot , \cdot \rangle_H)$ be a separable Hilbert space and let $X \in L^2(\Omega;H)$ be a strongly measurable random field over the probability space $(\Omega, \sigma(Y), P)$. Let $(p_k(Y))_{k \in \N_0}$ denote the polynomial chaos basis of $L^2(\Omega; \R)$, then the random field $X: \Omega \to H$ admits the PC expansion
\begin{align*}
X(\omega) = \sum_{k \in \N_0} x_k \, p_k(Y(\omega)),
\end{align*}
with convergence in $L^2(\Omega;H)$, and spectral coefficients 
\begin{align*}
x_k = \int_\Omega X(\omega) p_k(Y(\omega)) \, dP(\omega) \in H.
\end{align*}
Its truncation to a maximal total polynomial degree $P \in \N_0$ is given by
\begin{align*}
X^{(M)}(\omega) = \sum_{k = 0}^M x_k p_k(Y(\omega)), \quad \text{ for } M + 1 = \frac{(N + P)!}{N!P!}.
\end{align*}
Due to the construction of the polynomial basis in the image space of the underlying random variables, we can represent the random fields by deterministic functions, called deterministic representation
\begin{align*}
\bar{X}(y) = \sum_{k \in \N_0} x_k p_k(y), \quad y \in \Gamma,
\end{align*}
with convergence in $L^2(\Gamma; H)$, such that for $P$-almost $\omega \in \Omega$
\begin{align*}
X(\omega) = \bar{X}(Y(\omega)).
\end{align*}

\subsection{Stochastic Galerkin method}
In the first step, we expand the transformed matrix diffusion field $\Ar: \Omega \times \Dr \to \mathcal{M}_2^d$ and transformed forcing term $\Fr: \Omega \times \Dr \to \R$ into their respective polynomial chaos expansions. This section is mainly based on the works of Xiu and Tartakovsky \citep{Xiu2006} and Kundu et al.\ \citep{Kundu2014} for stochastic Galerkin methods for random domain problems. The structure of the section is based on \citep[Ch. 3.2]{Musco2024}.
We focus on the strong formulation of the transformed elliptic random PDE, Problem~\ref{Problem:StrongFormTransformedERPDE}, and the stochastic weak formulation of the transformed elliptic random PDE, Problem~\ref{Problem:StochasticWeakForm_TransformedERPDE}. The polynomial chaos expansions used for this are truncated at a maximal polynomial degree $P \in \N_0$. 

As a direct consequence of Lemma~\ref{Lemma:CourantFischer_Estimate_Frobenius_Aref}, we get
\begin{align*}
\int_\Omega \int_{\Dr} \lVert \Ar(\omega, x) \rVert_{F_2}^2 \, \dd
\leq d \overline{\sigma}^2 \lambda^d(\Dr) < \infty,
\end{align*}
and hence $\Ar \in L^2(\Omega; L^2(\Dr; \mathcal{M}_2^d))$. Therefore, Assumption~\ref{Assumption:RD_And_Trafo_C1} suffices to expand the random fields of the transformed stochastic weak formulation, Problem~\ref{Problem:StochasticWeakForm_TransformedERPDE}, into their respective polynomial chaos expansions. For the strong formulation of the transformed PDE, Problem~\ref{Problem:StrongFormTransformedERPDE}, we have to place additional assumptions on the random domains and the diffusion coefficient in order for the strong residual to be meaningful. 
\begin{assumption} \label{Assumption:AdditionalAssmStrongResidual}
In addition to Assumption~\ref{Assumption:Bounds_Det_Diffusivity} and \ref{Assumption:RD_And_Trafo_C1}, we assume that there is a uniform bound $C > 0$, such that
\begin{align*}
\lVert \T(\omega, \cdot) \rVert_{C^2(\DDr; \R^d)} \leq C, \quad \lVert \T^{-1}(\omega, \cdot) \rVert_{C^2(\DDw; \R^d)} \leq C
\end{align*}
for every $\omega \in \Omega$. Furthermore, we assume that the diffusion coefficient $a: \mathbf{D} \to \R$ is differentiable on $\mathbf{D}$, and $\nabla a \in L^{\infty}(\mathbf{D}; \R^d)$.
\end{assumption}
\begin{theorem} \label{Theorem:TransformedMatrixDiffusivityInH1}
Under Assumption~\ref{Assumption:AdditionalAssmStrongResidual}, the transformed matrix diffusion field satisfies $\Ar \in L^2(\Omega; H^1(\Dr; \mathcal{M}_2^d))$.
\end{theorem}
\begin{proof}
We show that
\begin{align*}
\lVert \Ar \rVert_{L^2(\Omega; H^1(\Dr; \mathcal{M}_2^d))}^2
&= \int_\Omega \int_{\Dr} \lVert \Ar(\omega, x) \rVert_{F_2}^2 
	+ \lVert D^1 \Ar (\omega, x) \rVert_{F_3}^2 \, \dd 
	\\
	& 
\leq d \overline{\sigma}^2 \lambda^d(\Dr) + \int_\Omega \int_{\Dr} \lVert D^1 \Ar(\omega, x) \rVert_{F_3}^2 \, \dd 
\end{align*}
is finite. For that it suffices to show that $\lVert D^1 \Ar(\omega, x) \rVert_{F_3}^2$ is uniformly bounded for every $\omega \in \Omega, x \in \DDr$. The derivative $D^1 \Ar(\omega, x) \in \mathcal{M}_3^d$ is a rank three tensor with entries
\begin{align*}
\big(D^1 \Ar(\omega, x) \big)_{ijk} 
	&=
\frac{\partial}{\partial x_k} \big(\Ar(\omega, x)\big)_{ij}
	\\
	&=
\Big(\frac{\partial}{\partial x_k} \big( \lvert \operatorname{det} J_\T (\omega, x) \rvert \widehat{a}(\omega, x) \big) \Big) \big( J_\T^{-1}(\omega, x) J_\T^{-T}(\omega, x) \big)_{ij}
	\\
	& \hspace*{0.5cm}
+ \lvert \operatorname{det} J_\T (\omega, x) \rvert \widehat{a}(\omega, x) \frac{\partial}{\partial x_k} \big( J_\T^{-1}(\omega, x) J_\T^{-T}(\omega, x) \big)_{ij},
\end{align*}
for $i,j,k = 1, \dots , d$. By Assumption~\ref{Assumption:AdditionalAssmStrongResidual}, the factor
\begin{align*}
\big| \frac{\partial}{\partial x_k} \bigg( \lvert \operatorname{det} J_\T(\omega, x) \rvert \widehat{a}(\omega, x) \big) \bigg| \leq C_1
\end{align*}
is bounded by a constant $C_1 > 0$ for every $k = 1, \dots , d, \, \omega \in \Omega, x \in \DDr$.
Therefore, we get
\begin{align*}
& \lVert D^1 \Ar(\omega, x) \rVert_{F_3}^2 
	\\ 
	&
= \sum_{i,j,k = 1}^d 
\bigg( 
\Big( \frac{\partial}{\partial x_k} \big( \lvert \operatorname{det} J_\T (\omega, x) \rvert \widehat{a}(\omega, x) \big) \Big) \big( J_\T^{-1}(\omega, x) J_\T^{-T}(\omega, x) \big)_{ij}
+ \lvert \operatorname{det} J_\T (\omega, x) \rvert \widehat{a}(\omega, x) \frac{\partial}{\partial x_k} \big( J_\T^{-1}(\omega, x) J_\T^{-T}(\omega, x) \big)_{ij}
\bigg)^2
	\\
	& 
\leq 2 d C_1^2 \lVert J_\T^{-1}(\omega, x) J_\T^{-T}(\omega, x) \rVert_{F_2}^2
+ 2 (\tau_+ a_+)^2 \sum_{i,j,k = 1}^d 
\bigg( \frac{\partial}{\partial x_k} \big( J_\T^{-1}(\omega, x) J_\T^{-T}(\omega, x) \big)_{ij} \bigg)^2 
	\\
	&
\leq 2 d^2 C_1^2 \sigma_+^2 
+ 2 (\tau_+ a_+)^2 \sum_{i,j,k = 1}^d 
\bigg( \frac{\partial}{\partial x_k} \big( J_\T^{-1}(\omega, x) J_\T^{-T}(\omega, x) \big)_{ij} \bigg)^2,
\end{align*}
for every $\omega \in \Omega, x \in \DDr$, where the uniform bound of the first term stems from Theorem~\ref{Thm:EstimateSmallestAndLargestSingularValue}. Now we consider the second term
\begin{align*}
&\sum_{i,j,k = 1}^d 
\bigg( \frac{\partial}{\partial x_k} \big( J_\T^{-1}(\omega, x) J_\T^{-T}(\omega, x) \big)_{ij} \bigg)^2
	=
\sum_{i,j,k = 1}^d 
\bigg( \frac{\partial}{\partial x_k} 
\sum_{l = 1}^d \big( J_\T^{-1}(\omega, x) \big)_{il} \big( J_\T^{-T}(\omega, x) \big)_{lj} 
 \bigg)^2
 	\\
 	&=
\sum_{i,j,k = 1}^d 
\bigg(
	\sum_{l = 1}^d \big( D^1 J_\T^{-1}(\omega, x) \big)_{ilk} \big( J_\T^{-T}(\omega, x) \big)_{lj} 
	+ \sum_{l = 1}^d \big( J_\T^{-1}(\omega, x) \big)_{il} \big( D^1 J_\T^{-T} (\omega, x) \big)_{ljk} 
\bigg)^2
	\\
	& \leq
2 \sum_{i,j,k = 1}^d
\bigg(
	\sum_{l = 1}^d \big( D^1 J_\T^{-1} (\omega, x) \big)_{ilk}^2
	\bigg)
\bigg(
	\sum_{l = 1}^d \big( J_\T^{-T}(\omega, x) \big)_{lj}^2
	\bigg)
+ 2 \sum_{i,j,k = 1}^d
\bigg(
	\sum_{l = 1}^d \big( J_\T^{-1}(\omega, x)\big)_{il}^2
	\bigg)
\bigg(
	\sum_{l = 1}^d \big( D^1 J_\T^{-T}  (\omega, x) \big)_{ljk}^2
	\bigg)
	\\
	&=
2 \big\lVert D^1 J_\T^{-1}  (\omega, x) \big\rVert_{F_3}^2 \big\lVert J_\T^{-T}(\omega, x) \big\rVert_{F_2}^2
	+ 2 \big\lVert J_\T^{-1}(\omega, x) \big\rVert_{F_2}^2 \big\lVert D^1 J_\T^{-T}(\omega, x) \big\rVert_{F_3}^2 \leq C_2,
\end{align*}
which is uniformly bounded by a constant $C_2 > 0$ for every $\omega \in \Omega, x \in \DDr$ by Assumption~\ref{Assumption:RD_And_Trafo_C1} and \ref{Assumption:AdditionalAssmStrongResidual}.
\end{proof}
Let $Y = (Y_1, \dots, Y_N): \Omega \to \Gamma$, $\Gamma \subset \R^N$ denote the multivariate driving random variable and consider the probability space $(\Omega, \sigma(Y), P)$.
 Furthermore, let $(p_k(Y))_{k \in \N_0}$ denote the polynomial chaos basis of $L^2(\Omega;\R)$ and accordingly let $(p_k)_{k \in \N_0}$ denote the polynomial chaos basis of $L^2(\Gamma; \R)$. 
The deterministic representation of the transformed matrix diffusion coefficient and the transformed forcing term read
\begin{align*}
\bar{A}_{\text{ref}}(y, x) = \sum_{k \in \N_0} A_k(x) \, p_k(y), \quad
\bar{f}_{\text{ref}}(y, x) = \sum_{k \in \N_0} f_k(x) \, p_k(y), \quad y \in \Gamma, \, x \in \Dr,
\end{align*}
where the spectral coefficients satisfy $A_k \in H^1(\Dr; \mathcal{M}_2^d)$ for the strong formulation of the transformed PDE, or $A_k \in L^2(\Dr; \mathcal{M}_2^d)$ for the stochastic weak formulation of the transformed PDE, and $f_k \in L^2(\Dr;\R)$ for every $k \in \N_0$. 
With these assumptions, and by the Doob-Dynkin Lemma (see \citep[Lem. 1.14]{Kallenberg2021}), the stochasticity of the solution $\Ur :\Omega \times \DDr \to \R$ of the transformed elliptic random PDE is a function of $Y: \Omega \to \Gamma$. 
We choose $L^2(\Omega; H_0^1(\Dr; \R) \cap H^2(\Dr; \R))$ as solution space for the strong form and $L^2(\Omega; H_0^1(\Dr; \R))$ as solution space for the weak form, and obtain the deterministic representation $\bar{u}_{\text{ref}}: \Gamma \times \DDr \to \R$ of the unknown solution with respect to  $(p_k)_{k \in \N_0}$, as
\begin{align*}
\bar{u}_{\text{ref}}(y, x) = \sum_{k \in \N_0} u_k(x) \, p_k(y), \quad y \in \Gamma, \, x \in \DDr,
\end{align*}
where $u_k \in H_0^1(\Dr; \R) \cap H^2(\Dr; \R)$ or $u_k \in H_0^1(\Dr; \R)$ for every $k \in \N_0$.
For the stochastic Galerkin approximation, the deterministic representations are truncated at a maximal polynomial degree of $P \in \N_0$. 
This results in the solution space 
\begin{align*}
L^{2;(M)}(\Gamma; H(\Dr;\R)) \cong L^{2;(M)}(\Gamma; \R) \otimes H(\Dr; \R),
\end{align*}
where $H(\Dr; \R) \in \{H_0^1(\Dr; \R), H_0^1(\Dr; \R) \cap H^2(\Dr; \R) \}$ for the respective formulations, and the finite dimensional space $L^{2;(M)}(\Gamma;\R)$ is spanned by the first $M + 1= \frac{(N + P)!}{N! P!}$ basis polynomials $(p_k)_{k \in \{ 0, \dots , M \} }$ of $L^2(\Gamma; \R)$.
The stochastic Galerkin method, in the strong-residual form, seeks to find an approximation $\tur: \Gamma \times \DDr \to \R$, such that the truncated, strong residual $\mathcal{R}_{\tur} : \Gamma \times \Dr \to \R$, defined by
\begin{align*}
\mathcal{R}_{\tur}(y,x) := \nabla \cdot \left( \tar (y,x) \nabla \tur (y,x) \right) + \tfr (y,x),
\end{align*}
is orthogonal to the space $L^{2;(M)}(\Gamma; \R)$ for $\lambda^d$-almost every $x \in \Dr$. Here,
\begin{align*}
\tar (y,x):= \sum_{k = 0}^M A_k(x) \, p_k(y) , \quad \tfr (y,x) := \sum_{k = 0}^M f_k(x) \, p_k(y)
\end{align*}
denote the truncated deterministic representation of the transformed matrix diffusion coefficient and transformed forcing term. 
The orthogonality is achieved by a Galerkin projection, which results in a system of $M + 1$, usually coupled, deterministic PDEs.
\begin{problem}[Strong form stochastic Galerkin approximation of the transformed problem] \label{Problem:StrongFormSGA}
For the truncated deterministic representation $\tfr \in L^{2;(M)}(\Gamma; L^2(\Dr; \R))$ of a transformed given forcing term, find the spectral coefficients $u_0, \dots , u_M \in H_0^1(\Dr; \R) \cap H^2(\Dr; \R)$ of $\tur (y,x) = \sum_{i = 0}^M u_i(x) \, p_i(y) \in L^{2;(M)}(\Gamma; H_0^1(\Dr; \R) \cap H^2(\Dr; \R))$ such that for every $ k = 0, \dots, M$
\begin{align*}
\big\langle \mathcal{R}_{\tur}(\cdot, x) \, , \, p_k \big\rangle_{L^2(\Gamma;\R)}
	&=
0 
	\qquad
 x \in \Dr, 
 	\\
\big\langle \tur (\cdot, x) \, , \, p_k \big\rangle_{L^2(\Gamma;\R)} 
	&=
0
	\qquad
 x \in \partial \Dr.
\end{align*}  
\end{problem}
For the stochastic Galerkin approximation of the weak formulation of the transformed elliptic random PDE, we define the truncated deterministic representation of the bilinear and linear form. By using the truncated deterministic representation of the involved fields, let
\begin{align*}
& \tB: \tLH \times \tLH \to \R,
	\\
	& \hspace*{0.9cm}
(u, v) \mapsto 
\int_\Gamma \int_{\Dr} \langle \tar (y,x) \, \nabla u (y,x) , \nabla v (y,x) \rangle_2 \, d \lambda^d(x) \, d\mu(y),
	\\
	&
\tF : \tLH \to \R,
	\\
	& \hspace*{0.9cm}
v \mapsto 
\int_\Gamma \int_{\Dr} \tfr (y, x) v (y, x)\, d\lambda^d(x) \, d\mu(y).
\end{align*}
\begin{problem}[Weak form stochastic Galerkin approximation of the transformed problem]
\label{Problem:WeakFormSGA}
For the truncated deterministic representation $\tfr \in L^{2;(M)}(\Gamma; L^2(\Dr; \R))$ of a transformed given forcing term, find the spectral coefficients $u_0, \dots , u_M \in H_0^1(\Dr; \R)$ of $\tur (y,x) = \sum_{i = 0}^M u_i(x) \, p_i(y) \in L^{2;(M)}(\Gamma; H_0^1(\Dr; \R))$ such that
\begin{align*}
\tB (\tur, v) = \tF (v) \text{ for all } v \in \tLH .
\end{align*}
\end{problem}

\begin{theorem} \label{Thm:RitzEnergyFunction}
There exists a unique solution $u^*\in \tLH$ of Problem \ref{Problem:WeakFormSGA}. Furthermore, the Ritz energy functional $\tE: \tLH \to \R$, defined by
\begin{align*}
\tE(u) := \frac{1}{2} \tB (u,u) - \tF (u),
\end{align*}
has a unique minimum. The corresponding minimizer 
\begin{align*}
u^* = \underset{u \in \tLH}{\operatorname{arg \, min}} \tE (u)
\end{align*}
is the unique stochastic Galerkin approximation of Problem \ref{Problem:StochasticWeakForm_TransformedERPDE}, i.e.\ the solution of Problem \ref{Problem:WeakFormSGA}.
\end{theorem}
We refer to \citep[Thm. 3.13]{Musco2024} for a proof. 
The approximation error introduced by truncating the random fields in the bilinear and linear form decays exponentially for the stochastic Galerkin approximation under some additional assumptions (see \citep[Ch. 6]{Babuska2004}, and \citep[Ch. 4.1]{Hakula2024} for a special case). 
If the transformed matrix diffusion coefficient $\Ar \in L^2(\Gamma; L^2(\Dr; \mathcal{M}_2^d))$ is fully resolved by the PC expansion, i.e.\ $\tar = \bar{A}_{\text{ref}}$, the approximation error vanishes, since $F(v) = \tF(v)$ for all $v \in \tLH$ due to the orthonormality of $(p_k)_{k \in \N_0}$.

\subsection{Computational form of the stochastic Galerkin systems}
We assemble the systems of equations derived in the previous section for the given stochastic Galerkin approximations (SGA). 
Starting with the strong form stochastic Galerkin approximation, Problem \ref{Problem:StrongFormSGA}, we note that for a matrix valued mapping $A: \R^d \to \mathcal{M}_2^d$ and a scalar field $u: \R^d \to \R$, the divergence of $A \nabla u$ can be represented by
\begin{align*}
\nabla \cdot (A(x) \nabla u(x)) 
	&=
\sum_{i = 1}^d \frac{\partial}{\partial x_i} \big( A(x) \nabla u(x) \big)_i
= \sum_{i = 1}^d \frac{\partial}{\partial x_i} \bigg( \sum_{j = 1}^d A_{ij}(x) \frac{\partial}{\partial x_j} u(x) \bigg)
	\\
	&=
\sum_{i,j = 1}^d \bigg(
\Big(\frac{\partial}{\partial x_i} A_{ij}(x) \Big) \frac{\partial}{\partial x_j} u(x) + A_{ij}(x) \frac{\partial^2}{\partial x_i \partial x_j} u(x) 
\bigg)
	\\
	&=
\sum_{j = 1}^d \bigg(
\sum_{i = 1}^d \frac{\partial}{\partial x_i} A_{ij}(x) \bigg) \frac{\partial}{\partial x_j} u(x)
+ \sum_{i,j = 1}^d A_{ij}(x) \frac{\partial^2}{\partial x_i \partial x_j} u(x) 
	\\
	&=
\langle \operatorname{div}(A(x)) , \nabla u(x) \rangle_2 + \langle A(x), H_u(x) \rangle_{F_2},
\end{align*}
where $\operatorname{div}(A(x)) \in \R^d$ denotes the column-wise divergence of $A(x)$, i.e.\ $(\operatorname{div}(A(x)))_j := \sum_{i=1}^d \frac{\partial}{\partial x_i} A_{ij}(x)$, and $H_u(x) \in \mathcal{M}_2^d$ denotes the Hessian of $u$. Using this, we get for $k = 0, \dots , M$:
\begin{align*}
\big\langle \mathcal{R}_{\tur}(\cdot, x) \, , \, p_k \big\rangle_{L^2(\Gamma;\R)}
	&=
\bigg\langle \nabla \cdot \bigg(
\Big( \sum_{i = 0}^M A_i(x) p_i \Big)
\Big( \sum_{j = 0}^M \nabla u_j(x) p_j \Big)
\bigg)
+ \sum_{n = 0}^M f_n(x) p_n
,
p_k
\bigg\rangle_{L^2(\Gamma; \R)}
	\\
	&=
\sum_{i,j = 0}^M \nabla \cdot (A_i(x) \nabla u_j(x))
\langle p_i p_j , p_k \rangle_{L^2(\Gamma; \R)} 
+ \sum_{n = 0}^M f_n(x) \langle p_n , p_k \rangle_{L^2(\Gamma; \R)}
	\\
	&=
\sum_{i,j = 0}^M 
\big( \langle \operatorname{div}(A_i(x)), \nabla u_j(x) \rangle_2 + \langle A_i(x) , H_{u_j}(x) \rangle_{F_2} \big) \langle p_i p_j , p_k \rangle_{L^2(\Gamma; \R)} 
+ f_k(x) \quad && x \in \Dr,
	\\
\big\langle \tur, p_k \big\rangle_{L^2(\Gamma; \R)} 
	&=
\bigg\langle \sum_{i = 0}^M u_i(x) p_i, p_k \bigg\rangle_{L^2(\Gamma; \R)} = u_k(x) 
&& x \in \partial \Dr.
\end{align*}
By defining the Galerkin tensor $G = (G_{ijk}) \in \mathcal{M}_3^{M+1}$, 
the tensor operator $\mathbf{A}: \Dr \to \R^{(M+1) \times (M+1) \times d}, \, \mathbf{A}(x) = (\mathbf{A}_{jk}(x))$, 
and the tensor operator $\mathbf{B}: \Dr \to \R^{(M+1) \times (M+1) \times d \times d}, \, \mathbf{B}(x) = (\mathbf{B}_{jk}(x))$ via
\begin{align*}
G_{ijk} := \langle p_i p_j, p_k \rangle_{L^2(\Gamma; \R)} \in \R,
\quad
\mathbf{A}_{jk}(x) := \sum_{i = 0}^M \operatorname{div}(A_i(x)) G_{ijk} \in \R^d,
\quad
\mathbf{B}_{jk}(x) := \sum_{i = 0}^M A_i(x) G_{ijk} \in \R^{d \times d},
\end{align*}
the stochastic Galerkin system can be written in the following way.
\begin{problem}[Computational form of the strong form SGA] \label{Problem:ComputationalFormStrongSGA}
For the given operators defined above, find $u_0, \dots , u_M \in H_0^1(\Dr; \R) \cap H^2(\Dr; \R)$ such that
\begin{align*}
- \sum_{j = 0}^M \big( \langle \mathbf{A}_{jk}(x) , \nabla u_j(x) \rangle_2
+ \langle \mathbf{B}_{jk}(x), H_{u_j}(x) \rangle_{F_2} \big) 
	&=
f_k(x)
	&& \hspace*{-2cm}
 x \in \Dr, \, k = 0, \dots, M,
 	\\
u_k(x) 
	&= 
0
	 && \hspace*{-2cm}
x \in \partial \Dr, \, k = 0, \dots , M.
\end{align*}
\end{problem}
We continue with the weak form stochastic Galerkin approximation. Let $u,v \in \tLH$ with $u(y,x) = \sum_{i = 0}^M u_i(x) p_i(y), \, v(y,x) = \sum_{i=0}^M v_i(x) p_i(y)$. With the operators defined above, we get 
\begin{align*}
\tB(u,v) 
	&=
\int_\Gamma \int_{\Dr} \bigg\langle 
\bigg(\sum_{i=0}^M A_i(x) p_i(y) \bigg) \bigg(\sum_{j=0}^M \nabla u_j(x) p_j(y) \bigg)
,
\sum_{k=0}^M \nabla v_k(x) p_k(y) 
\bigg\rangle_2 \, \ddm
	\\
	&=
\int_{\Dr} \sum_{j,k=0}^M \langle \mathbf{B}_{jk}(x) \nabla u_j(x) , \nabla v_k(x) \rangle_2 \, d\lambda^d(x),
	\\
\tF (v) 
	&=
\int_\Gamma \int_{\Dr} 
\bigg( \sum_{i=0}^M f_i(x) p_i(y) \bigg) \bigg(\sum_{j = 0}^M v_j(x) p_j(y) \bigg) \ddm
	\\
	&=
\int_{\Dr} \sum_{i=0}^M f_i(x) v_i(x) \, d \lambda^d(x).
\end{align*}
\begin{problem}[Computational form of the weak form SGA] \label{Problem:ComputationalFormWeakSGA}
For the given operators defined above, find $u_0, \dots , u_M \in H_0^1(\Dr; \R)$ such that for every $v = \sum_{i = 0}^M v_i(x) p_i(y) \in \tLH$:
\begin{align*}
\int_{\Dr} \sum_{i,j = 0}^M \langle \mathbf{B}_{ij}(x) \nabla u_i(x) , \nabla v_j(x) \rangle_2 \, d\lambda^d(x)
	&= 
\int_{\Dr} \sum_{k = 0}^M f_k(x) v_k(x) \, d\lambda^d(x).
\end{align*}
\end{problem}
\begin{remark} \label{Remark:EnergyMinimizationWeakSGA}
Combining Theorem \ref{Thm:RitzEnergyFunction} with the computational form, the solution $u^* \in \tLH$, of Problem \ref{Problem:ComputationalFormWeakSGA}, is given as
\begin{align*}
u^* = \underset{u \in \tLH}{\operatorname{arg \, min}} 
\int_{\Dr} \frac{1}{2}  \sum_{i,j = 0}^M \langle \mathbf{B}_{ij}(x) \nabla u_i(x) , \nabla u_j(x) \rangle_2 - \sum_{k=0}^M f_k(x) u_k(x) \, d\lambda^d(x).
\end{align*}
\end{remark}
The computational forms are the basis for the training strategies of the neural networks, developed in the next section.

\section{Deep learning approach} \label{sec:DL}
In \citep{Musco2024}, we developed the neural network \textit{S-GalerkinNet} to solve the strong stochastic Galerkin approximation of an elliptic random PDE, as well as the neural network \textit{S-RitzNet} to solve the weak stochastic Galerkin approximation of an elliptic random PDE. We used the neural networks as surrogates for the spectral coefficients in the respective stochastic Galerkin formulations. 
A theoretical justification for that is given by the so-called \textit{universal approximation property}, stating that feedforward neural networks (in different setups) can approximate any continuous function on a bounded domain up to arbitrary precision. Here we mention the approximation property for sigmoidal and more general activation functions, as used in this work (see e.g.\ \citep{Cybenko1989, Hornik1989, Mhaskar1993}).
We adapt the architectures, developed in \citep{Musco2024}, to solve the problems addressed in this work. For the sake of completeness, we also introduce the neural networks in the following.
\subsection{Definition and notation}
In this section we recap the notation and basic setup of the neural networks, as given in \citep[Ch. 4]{Musco2024} and adjust the computational architectures of the respective neural networks to the random domain problems considered in this work. We approximate the spectral coefficients $(u_0, \dots, u_M)^T: \overline{\DD} \to \R^{M+1}$ of the stochastic Galerkin approximation $\tu (y,x) = \sum_{i = 0}^M u_i(x) p_i(y)$ of Problem \ref{Problem:ComputationalFormStrongSGA} and \ref{Problem:ComputationalFormWeakSGA} respectively by deep feedforward neural networks. 

A deep feedforward neural network of depth $D \in \N$ defines a parameterized mapping $\mathcal{N}_{\theta} : \R^{d_0} \to \R^{d_{D+1}}$, with parameters $\theta \in \Theta$, that takes an input $x \in \R^{d_0}$, and consists of $D$ hidden layers $\mathcal{L}^{(1)}, \dots , \mathcal{L}^{(D)}$ of sizes $d_1, \dots , d_D \in \N$, and an output layer $\mathcal{L}^{(D+1)}$ of size $d_{D+1}$. 
The size of a layer determines the number of computational nodes, called neurons, it consists of. Each neuron computes an affine transformation of the output of the previous layer, respectively the input, which is then composed, called activated, with a usually non-linear function called activation function. Formally, the output $l^{(i)} \in \R^{d_i}$ of the $i$-th layer $\mathcal{L}^{(i)} : \R^{d_{i-1}} \to \R^{d_i}$, for $i = 1, \dots, D+1$, is given by
\begin{align*}
l^{(i)} = \mathcal{L}^{(i)}(l^{(i-1)}) := \sigma^{(i)} \left( W^{(i)} l^{(i-1)} + b^{(i)} \right),
\end{align*}
where $\sigma^{(i)}: \R \to \R$ is the activation function of the $i$-th layer, applied component-wise, $W^{(i)} \in \R^{d_i \times d_{i-1}}$ and $b^{(i)} \in \R^{d_i}$ denote the weight matrix and the bias of the $i$-th layer, and $l^{(i-1)} \in \R^{d_{i-1}}$ the output of the $(i-1)$-th layer, respectively the input $x \in \R^{d_0}$ to the neural network for $i - 1 = 0$. The $j$-th component of $l^{(i)} = (l_1^{(i)} , \dots , l_{d_i}^{(i)})$ is the output of the $j$-th neuron in the $i$-th layer.
The deep feedforward neural network is defined as the composition of the layers
\begin{align*}
\mathcal{N}_{\theta} : \R^{d_0} \to \R^{d_{D+1}}, \, x \mapsto \left( \mathcal{L}^{(D+1)} \circ \mathcal{L}^{(D)} \circ \dots \circ \mathcal{L}^{(1)} \right)(x),
\end{align*}
where the parameters, called weights, are given by
\begin{align*}
\theta = \left( W^{(1)}, b^{(1)}, \dots , W^{(D+1)}, b^{(D+1)} \right) \in \Theta \subset \R^{d_1 \times d_0} \times \R^{d_1} \times \dots \times \R^{d_{D+1} \times d_D} \times \R^{d_{D+1}} := \mathcal{R}.
\end{align*}
The parameter space $\Theta$ can be a proper subset of $\mathcal{R}$, placing additional constraints on the weights, e.g.\ non-negativity or special weight matrices resulting in disconnected neurons. If smooth activation functions are used, the neural network is, as a composition of smooth functions, itself smooth.
Neural networks have a very high expressive power and achieve an impressive performance in a wide range of tasks. The task is usually taught to the neural network by a specified loss function
penalizing wrong predictions in a proper way. 
The algorithm then seeks to find a set of weights that minimizes some norm or transformation of the loss function, called risk. In supervised learning, the optimization is carried out with the help of a dataset containing ground-truth data, whereas in unsupervised learning no such dataset is used, the loss is calculated directly via a transformation of the input, without the need of ground truth labels. This can be advantageous, when training labels are very expensive or not available. Combinations of both training methods are also very common.

\subsection{Neural network architecture and loss function}
In this work we adapt the \textit{S-GalerkinNet}, developed in \citep{Musco2024}, to approximate a solution of the strong form stochastic Galerkin approximation of the transformed problem, Problem \ref{Problem:StrongFormSGA}, and the \textit{S-RitzNet} to approximate a solution of the weak form stochastic Galerkin approximation of the transformed problem, Problem \ref{Problem:WeakFormSGA}. 
The training strategies, namely the corresponding loss functions, are derived from the computational forms, Problem \ref{Problem:ComputationalFormStrongSGA} for the \textit{S-GalerkinNet}, and Problem \ref{Problem:ComputationalFormWeakSGA} for the \textit{S-RitzNet}, using the Ritz energy minimization described in Remark \ref{Remark:EnergyMinimizationWeakSGA}. Furthermore, a schematic overview of the architectures of the neural networks is given. 

Starting with the \textit{S-GalerkinNet}, the goal is to  approximate the spectral coefficients $u_0, \dots , u_M \in H_0^1(\Dr; \R) \cap H^2(\Dr; \R)$ of the stochastic Galerkin approximation $\tur (y,x) = \sum_{i = 0}^M u_i(x) p_i(y) \in L^{2;(M)}(\Gamma; H_0^1(\Dr; \R) \cap H^2(\Dr; \R))$, solving Problem \ref{Problem:StrongFormSGA}, by the neural network.
Here, we don't approximate the spectral coefficients $(u_0, \dots , u_M)^T: \DDr \to \R^{M+1}$ directly by the \textit{S-GalerkinNet} $\mathcal{N}_\theta^{\operatorname{SG}} : \DDr \to \R^{M+1}$, but we incorporate the homogeneous Dirichlet boundary conditions as a hard constraint. For that, we use an enforcer function $e \in C^2(\DDr; \R)$, such that $e(x) = 0$ for $x \in \partial \Dr$, and $e(x) > 0$ for $x \in \Dr$.
With that, we define the function
\begin{align*}
{\scriptstyle \mathcal{U}}_{\theta^*}^{\operatorname{SG}} = \big({\scriptstyle \mathcal{U}}_{0;\theta^*}^{\operatorname{SG}}, \dots , {\scriptstyle \mathcal{U}}_{M;\theta^*}^{\operatorname{SG}}\big)^T: \DDr \to \R^{M+1}, \, x \mapsto e(x) \cdot \mathcal{N}_{\theta^*}^{\operatorname{SG}}(x),
\end{align*}
and use the components as approximation of the spectral coefficients of the strong form stochastic Galerkin approximation $\tur \in L^{2;(M)}(\Gamma; H_0^1(\Dr; \R) \cap H^2(\Dr; \R))$, i.e.\
\begin{align*}
\tur (y,x) \approx \sum_{i = 0}^M {\scriptstyle \mathcal{U}}_{i;\theta^*}^{\operatorname{SG}}(x) p_i(y) =: \tu_{\text{ref}; \,\theta^*}(y,x).
\end{align*}
Note that enforcing the boundary conditions as a hard constraint is not limited to homogeneous Dirichlet boundary conditions or rectangular domains (see e.g.\ \citep{Berrone2023, Sukumar2022}). 
To derive the learning strategy of the \textit{S-GalerkinNet}, i.e.\ the optimization scheme to find the optimal parameters $\theta^* \in \Theta^{\operatorname{SG}}$, we recall the goal to find the neural network approximation ${\scriptstyle \mathcal{U}}_{\theta^*}^{\operatorname{SG}}: \DDr \to \R^{M+1}$ of the spectral coefficients, such that for every $k = 0 , \dots , M$ and $x \in \Dr$
\begin{align*}
\big\langle \mathcal{R}_{\tu_{\text{ref}; \, \theta^*}}(\cdot, x) , p_k \big\rangle_{L^2(\Gamma; \R)} \overset{!}{=} 0. 
\end{align*}
In order to define a proper loss function to achieve this goal, we note the following:
\begin{theorem}
For smooth activation functions, the mapping
\begin{align*}
x \mapsto \big\langle \mathcal{R}_{\tu_{\text{ref}; \, \theta}}(\cdot, x) , p_k \big\rangle_{L^2(\Gamma; \R)}
\end{align*}
is a member of $L^2(\Dr; \R)$ for every $\theta \in \Theta^{\operatorname{SG}}$ and every $k = 0, \dots , M$.
\end{theorem}
\begin{proof}
For smooth activation functions, the neural network $\mathcal{N}_\theta^{\operatorname{SG}} : \DDr \to \R^{M+1}$ is itself a smooth function, and therefore ${\scriptstyle \mathcal{U}}_{\theta}^{\operatorname{SG}} \in C^2(\DDr; \R^{M+1})$ for every $\theta \in \Theta^{\operatorname{SG}}$. Furthermore, the column-wise divergence $\operatorname{div}(A(x)) \in \R^d$ of a matrix valued mapping $A: \R^d \to \mathcal{M}_2^d$, satisfies
\begin{align*}
\lVert \operatorname{div}(A(x)) \rVert_2^2
	&=
\sum_{j = 1}^d \bigg( \sum_{i = 1}^d \frac{\partial}{\partial x_i} (A(x))_{ij} \bigg)^2
	\leq
d \sum_{i,j=1 }^d \Big( \frac{\partial}{\partial x_i} (A(x))_{ij} \Big)^2
	\leq
d \sum_{i,j,k=1}^d \Big(\frac{\partial}{\partial x_k} (A(x))_{ij} \Big)^2
	=
d \lVert D^1 A(x) \rVert_{F_3}^2.
\end{align*}
Therefore, we get for the projected residual
\begin{align*}
& \int_{\Dr} \big\lvert \langle \mathcal{R}_{\tu_{\text{ref}; \, \theta}}(\cdot, x) , p_k \big\rangle_{L^2(\Gamma; \R)} \big\rvert^2 \, d\lambda^d(x)
	\\
	& \hspace*{0.5cm} =
\int_{\Dr} \big\lvert
\sum_{i,j=0}^M 
\big(
\big\langle \operatorname{div}(A_i(x)) , \nabla {\scriptstyle \mathcal{U}}_{j;\theta}^{\operatorname{SG}}(x) \big\rangle_2
+
\big\langle A_i(x) , H_{{\scriptstyle \mathcal{U}}_{j;\theta}^{\operatorname{SG}}}(x) \big\rangle_{F_2}
\big)
\langle p_i p_j , p_k \rangle_{L^2(\Gamma; \R)}
+ f_k(x)
\big\rvert^2 \, d\lambda^d(x)
	\\
	& \hspace*{0.5cm} \leq
2(M + 1)^2 \sum_{i,j=0}^M 
\int_{\Dr} \big\lvert
\big(
\big\langle \operatorname{div}(A_i(x)) , \nabla {\scriptstyle \mathcal{U}}_{j;\theta}^{\operatorname{SG}}(x) \big\rangle_2
+
\big\langle A_i(x) , H_{{\scriptstyle \mathcal{U}}_{j;\theta}^{\operatorname{SG}}}(x) \big\rangle_{F_2}
\big)
\langle p_i p_j , p_k \rangle_{L^2(\Gamma; \R)}
\big\rvert^2 \, d \lambda^d(x)
+ 2 \lVert f_k \rVert_{L^2(\Dr, \R)}^2
	\\
	& \hspace*{0.5cm} \leq
4(M + 1)^2 \sum_{i,j=0}^M \lVert p_i p_j \rVert_{L^2(\Gamma; \R)}^2
\int_{\Dr}
\big\lvert
\big\langle \operatorname{div}(A_i(x)) , \nabla {\scriptstyle \mathcal{U}}_{j;\theta}^{\operatorname{SG}}(x) \big\rangle_2
\big\rvert^2
+
\big\lvert
\big\langle A_i(x) , H_{{\scriptstyle \mathcal{U}}_{j;\theta}^{\operatorname{SG}}}(x) \big\rangle_{F_2}
\big\rvert^2 \, d \lambda^d(x)
+ 2 \lVert f_k \rVert_{L^2(\Dr, \R)}^2
	\\
	& \hspace*{0.5cm} \leq
4(M + 1)^2 \sum_{i,j=0}^M \lVert p_i p_j \rVert_{L^2(\Gamma; \R)}^2
\int_{\Dr}
d \lVert D^1 A_i(x) \rVert_{F_3}^2 \lVert \nabla {\scriptstyle \mathcal{U}}_{j;\theta}^{\operatorname{SG}}(x) \rVert_2^2
+ \lVert A_i(x) \rVert_{F_2}^2 \lVert H_{{\scriptstyle \mathcal{U}}_{j;\theta}^{\operatorname{SG}}}(x) \rVert_{F_2}^2 \, d\lambda^d(x)
+ 2 \lVert f_k \rVert_{L^2(\Dr, \R)}^2
	\\
	& \hspace*{0.5cm} \leq
4(M + 1)^2 \sum_{i,j=0}^M \lVert p_i p_j \rVert_{L^2(\Gamma; \R)}^2
C_j \lVert A_i \rVert_{H^1(\Dr; \mathcal{M}_2^d)}^2 
+ 2 \lVert f_k \rVert_{L^2(\Dr, \R)}^2 
< \infty,
\end{align*}
which is finite by Theorem~\ref{Theorem:TransformedMatrixDiffusivityInH1}, and the corresponding polynomial chaos expansion. Furthermore, the constants are given by
\begin{align*}
C_j := \max \Big\lbrace d \max_{x \in \DDr} \lVert \nabla {\scriptstyle \mathcal{U}}_{j;\theta}^{\operatorname{SG}}(x) \rVert_2^2 , \max_{x \in \DDr} \lVert H_{{\scriptstyle \mathcal{U}}_{j;\theta}^{\operatorname{SG}}}(x) \rVert_{F_2}^2
\Big\rbrace,
\end{align*}
which are finite for every $\theta \in \Theta^{\operatorname{SG}}, \, j = 0, \dots , M$.
\end{proof}
This enables us to define an unsupervised loss function of least-squares-type to train the neural network, namely
\begin{align*}
\mathbf{L}^{\operatorname{SG}}(x; \mathcal{N}_\theta^{\operatorname{SG}}) := \frac{1}{M + 1} \sum_{k=0}^M \big\langle \mathcal{R}_{\tu_{\text{ref}; \, \theta}}(\cdot, x) , p_k \big\rangle_{L^2(\Gamma; \R)}^2.
\end{align*}
Working with the computational form of the strong form stochastic Galerkin approximation, Problem~\ref{Problem:ComputationalFormStrongSGA}, the associated risk is the Monte Carlo integration of the above defined loss function, i.e.\ the training procedure of the \textit{S-GalerkinNet} is given as follows.
\begin{problem}[Training of the \textit{S-GalerkinNet}] \label{Problem:Training_SGN}
For given training points $(x_1, \dots , x_n) \in \Dr^n$, find a set of parameters $\theta^* \in \Theta^{\operatorname{SG}}$, such that
\begin{align*}
& \lambda^d(\Dr) \, \frac{1}{n} \sum_{l = 1}^n \mathbf{L}^{\operatorname{SG}}(x_l; \mathcal{N}_{\theta^*}^{\operatorname{SG}})
	\\
	& \hspace*{0.5cm} =
\lambda^d(\Dr) \, \frac{1}{n} \sum_{l = 1}^n 
\frac{1}{M+1} \sum_{k=0}^M 
\left( 
\sum_{j=0}^M \big( 
\big\langle \mathbf{A}_{jk}(x_l) , \nabla {\scriptstyle \mathcal{U}}_{j;\theta^*}^{\operatorname{SG}}(x_l) \big\rangle_2
+ \big\langle \mathbf{B}_{jk}(x_l), H_{{\scriptstyle \mathcal{U}}_{j;\theta^*}^{\operatorname{SG}}}(x_l) \big\rangle_{F_2}
\big)
+ f_k(x_l)
\right)^2 
\overset{!}{=} 0.
\end{align*}
\end{problem}
For the stochastic Galerkin approximation $\tur \in \tLH$ of the weak form, we approximate the spectral coefficients $u_0, \dots , u_M \in H_0^1(\Dr ; \R)$ by the \textit{S-RitzNet} $\mathcal{N}_\theta^{\operatorname{SR}}: \DDr \to \R^{M+1}$. The training strategy of this neural network is based on the energy minimization of the computational form given in Remark~\ref{Remark:EnergyMinimizationWeakSGA}.
Since this defines a minimization of an integral operator over the spatial domain, choosing the integrand as loss function is quite natural for this task. 
We use an analogous notation to denote the neural network spectral coefficients with enforced Dirichlet boundary conditions by
\begin{align*}
{\scriptstyle \mathcal{U}}_{\theta}^{\operatorname{SR}} = 
({\scriptstyle \mathcal{U}}_{0;\theta}^{\operatorname{SR}}, \dots , {\scriptstyle \mathcal{U}}_{M;\theta}^{\operatorname{SR}})^T: \DDr \to \R^{M+1}, \, x \mapsto e(x) \cdot \mathcal{N}_{\theta}^{\operatorname{SR}}(x).
\end{align*}
Using the computational form, we define the loss function for the \textit{S-RitzNet} as 
\begin{align*}
\mathbf{L}^{\operatorname{SR}}(x; \mathcal{N}_\theta^{\operatorname{SR}}) 
:= 
\frac{1}{2} \sum_{i,j = 0}^M 
\big\langle \mathbf{B}_{ij}(x) \nabla {\scriptstyle \mathcal{U}}_{i ; \theta}^{\operatorname{SR}} (x) , \nabla {\scriptstyle \mathcal{U}}_{j ; \theta}^{\operatorname{SR}} (x) \big\rangle_2
- \sum_{k=0}^M f_k(x) {\scriptstyle \mathcal{U}}_{k ; \theta}^{\operatorname{SR}} (x).
\end{align*}
The corresponding risk is the Monte Carlo integration of the loss function, i.e.\ the training procedure of the \textit{S-RitzNet} is given as follows.
\begin{problem}[Training of the \textit{S-RitzNet}] \label{Problem:TrainingSRN}
For given training points $(x_1, \dots , x_n) \in \Dr^n$, find a set of parameters $\theta^* \in \Theta^{\operatorname{SR}}$, such that
\begin{align*}
&\lambda^d(\Dr) \frac{1}{n} \sum_{l = 1}^n \mathbf{L}^{\operatorname{SR}}(x_l; \mathcal{N}_{\theta^*}^{\operatorname{SR}}) 
	=
\lambda^d(\Dr) \frac{1}{n} \sum_{l = 1}^n 
	\Big( 
		\frac{1}{2} \sum_{i,j = 0}^M 
\big\langle \mathbf{B}_{ij}(x_l) \nabla {\scriptstyle \mathcal{U}}_{i ; \theta^*}^{\operatorname{SR}} (x_l) , \nabla {\scriptstyle \mathcal{U}}_{j ; \theta^*}^{\operatorname{SR}} (x_l) \big\rangle_2
- \sum_{k=0}^M f_k(x_l) {\scriptstyle \mathcal{U}}_{k ; \theta^*}^{\operatorname{SR}} (x_l)
	\Big)
	\\
	& \qquad =
\min_{\theta \in \Theta^{\operatorname{SR}}} \lambda^d(\Dr) \frac{1}{n} \sum_{l = 1}^n \mathbf{L}^{\operatorname{SR}}(x_l; \mathcal{N}_{\theta}^{\operatorname{SR}}).
\end{align*}
\end{problem}
Since the actual minimum is not known, the value of the risk is not as interpretable as for the strong residual. Therefore, every few epochs, we calculate a validation error to get a clearer quantification of the risk, and to formulate a stopping criterion for the neural network optimization. 

	\begin{figure}[t!]
	\begin{center}
	\includegraphics[width=0.9999\textwidth]{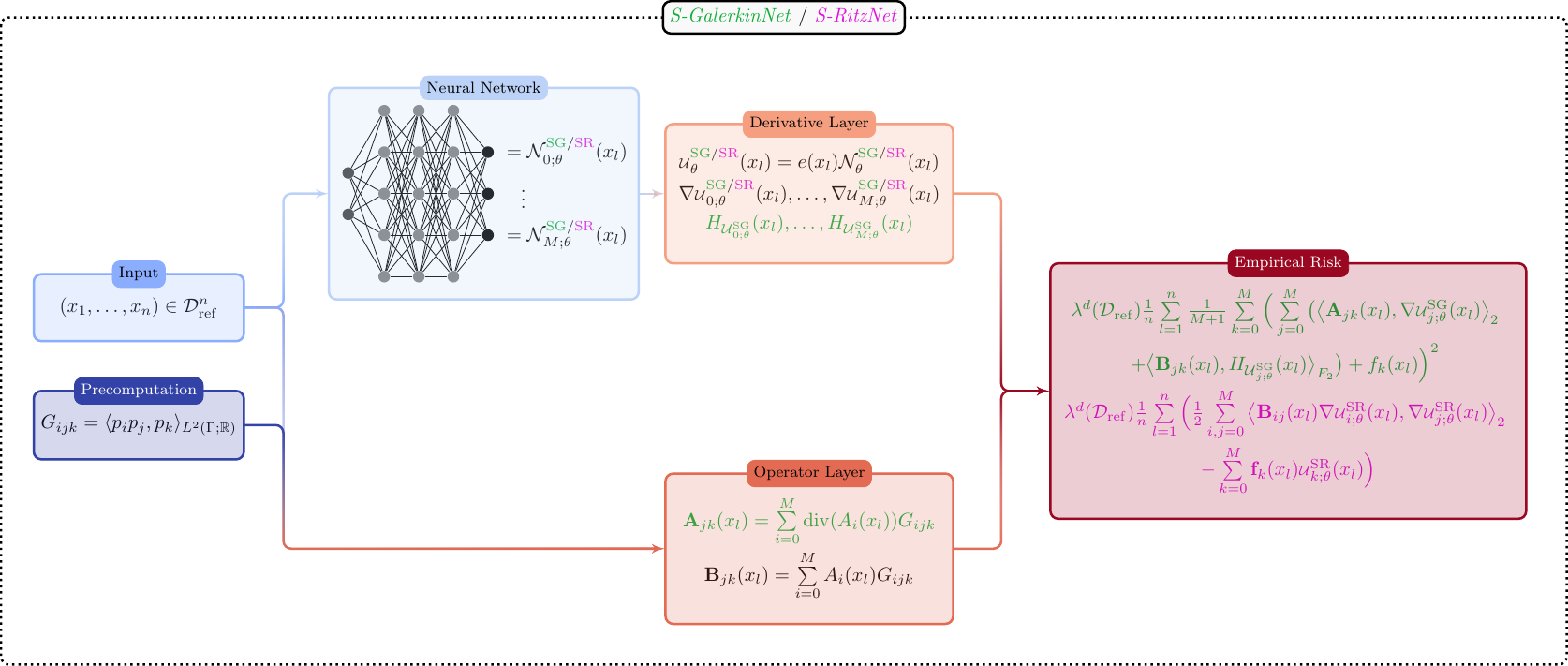}
	\end{center}
	\caption{Architecture of \textit{S-GalerkinNet} and \textit{S-RitzNet}.} \label{Fig:NN_Architecture_Combined}
	\end{figure}

A schematic overview of the architectures of the \textit{S-GalerkinNet} and the \textit{S-RitzNet} is given in Figure \ref{Fig:NN_Architecture_Combined}. The nodes written in black are the same for both architectures (with different weights), while the ones depicted in green only refer to the \textit{S-GalerkinNet}, and the ones depicted in magenta only refer to the \textit{S-RitzNet}. 
\\
The choice of input spatial points $(x_1, \dots , x_n) \in \Dr^n$ is determined experimentally for the specific experiment. In the experiments presented here, we work with quasi-random Sobol points (see \citep{Sobol1967}) and different quadrature points stated in the specific experiment.
The stochastic integrals, i.e.\ the Galerkin tensor $G \in \R^{(M+1) \times (M+1) \times (M+1)}$, are precomputed once in the offline phase.
\\
The fully connected neural networks consist of a feedforward structure and $M + 1$ output neurons, where the $k$-th output neuron $\mathcal{N}_{k;\theta}^{\operatorname{SG} / \operatorname{SR}} : \DDr \to \R$, for $k = 0 , \dots , M$, is used for the neural network approximation ${\scriptstyle \mathcal{U}}_{k ; \theta}^{\operatorname{SG} / \operatorname{SR}} = e(x)\mathcal{N}_{k;\theta}^{\operatorname{SG} / \operatorname{SR}}(x)$ of the $k$-th spectral coefficient of the stochastic Galerkin approximation 
\begin{align*}
\sum_{i = 0}^M {\scriptstyle \mathcal{U}}_{i;\theta}^{\operatorname{SG} / \operatorname{SR}}(x) p_i(y) \approx \tur (y,x).
\end{align*}
The specific details of the neural networks, such as depth, width, activation functions, etc.\ are determined experimentally, and stated in the corresponding experiment section.
\\
In the derivative layer, which is non-trainable, the boundary conditions are enforced, as described above, and the occurring derivatives of the neural networks are calculated using automatic differentiation (see \citep{Baydin2017}). 
In the also non-trainable operator layer, the operators for the computational forms are assembled, and the derivatives of the transformed matrix diffusion coefficient are either obtained analytically or via automatic differentiation. 
The additional computations for the \textit{S-GalerkinNet} in these layers impact the overall time-to-error performance of this approach as shown in the following experiments.
\\
In the last forward step, the above derived empirical risk is calculated. 
In the not depicted backward step, this risk is minimized and the neural network weights are updated, called training. The training of the \textit{S-GalerkinNet} and \textit{S-RitzNet} is purely unsupervised, i.e.\ no labeled ground-truth data is used. The optimization is  based on a variant of stochastic gradient descent, called Adam algorithm (see \citep{Kingma2015}), where gradients are calculated using the backpropagation algorithm. 
We refer to \citep[Ch. 6]{Goodfellow2016} for details on feedforward neural networks and backpropagation. 
During the training, a callback method monitors a validation error for a prescribed number of optimizer iterations. The neural network parameters corresponding to the best validation error performance during training are retained.

\section{Numerical Experiments} \label{Section:Numerical_Experiments}
In this section, we present numerical examples, where we solve two different random domain problems for the elliptic random partial differential equation. The numerical experiments are based on the transformed strong and weak stochastic Galerkin formulations for random domain problems. 
We approximate the solution of the stochastic Galerkin systems by the corresponding deep learning approaches \textit{S-GalerkinNet} and \textit{S-RitzNet}, introduced in Section~\ref{sec:DL}.
The neural networks are implemented using the Python library TensorFlow (\citep{Abadi2015}), and trained on a single \textit{NVIDIA RTX 3070} GPU and a \textit{12th Gen Intel Core i9-12900} CPU.
The error $\epsilon_{M, \theta}$, stated in the experiments, is an approximation of the relative $L^2(\Omega; H^1(\Dw; \R))$ distance between a reference solution $u : \Omega \times \Dw \to \R$ on the random domain (which is analytically available or obtained by other methods, stated in the experiments), and the neural network stochastic Galerkin approximation $u_{\theta}^{(M)} : \Omega \times \Dw \to \R, \, 
u_{\theta}^{(M)} (\omega, \mathcal{T}^{-1}(\omega, x)) = \sum_{i = 0}^M {\scriptstyle \mathcal{U}}_{i ; \theta}^{\operatorname{SG / SR}}(\mathcal{T}^{-1}(\omega, x)) p_i(Y(\omega))$, given by the \textit{S-GalerkinNet} respectively the \textit{S-RitzNet}, defined on the reference domain and transformed on the random domain. The error is given by
\begin{align*}
\epsilon_{M, \theta} 
	&\approx 
\frac{\big\lVert u - u_{\theta}^{(M)} \big\rVert_{L^2(\Omega; H^1(\Dw; \R))}}{
\lVert u \rVert_{L^2(\Omega; H^1(\Dw; \R))}
},
\end{align*}
where the derivatives are calculated analytically, respectively via automatic differentiation, the spatial integrals are approximated using the trapezoid rule, and the stochastic integrals are approximated using Monte Carlo integration. 

In addition to the approximation errors, we investigate the computational cost of the different deep learning approaches. The reported training times, presented in the performance plots encompass the complete Python wall clock time of the online phase of the respective training procedure. This includes all CPU and GPU computations for the complete training process as well as TensorFlow's graph tracing and initialization. 
We specifically compare the computational times of the strong residual based deep learning approaches with those of their weak Ritz energy based counterparts. 
To gain a more detailed understanding of the influence of the architectural differences between the strong and weak deep learning approaches, we compare specific components of a training step. These components consist of the forward pass together with the evaluation of the required spatial derivatives of the neural networks, the evaluation of the coupling operators together with the assembly of the loss function and the computation of the empirical risk, and, finally, the evaluation of the neural network derivatives with respect to the trainable parameters together with the Adam optimizer step. 
Since parts of the training process, such as the generation and preparation of the spatial grid and other auxiliary operations are performed on the CPU, while GPU operations are executed in parallel, the computational costs of the individual stages cannot be separated reliably using Python wall clock measurements. 
We therefore use TensorBoard, the profiling toolkit provided with TensorFlow, to profile the GPU execution times of the mentioned training step components, executed subsequently for profiling purposes.
The reported percentages for the individual components of a training step are based on these measured GPU execution times.

At first we consider an example in one spatial dimension on a randomly stretched interval, with an analytical solution available. 
To this end, let the deterministic reference domain be given by $\DDr = [0,1]$, and let the random domain mapping be given by $\mathcal{T}(\omega, x) = ( 0.1 + 5 \lvert Y(\omega) \rvert ) x$, where $Y \sim \mathcal{U}([-1,1])$. This yields randomly stretched intervals with random right boundary, i.e. realizations of the random domains are given by $\DDw = \{x \in \R \, : \, 0 < x < 0.1 + 5 \lvert Y(\omega) \rvert \}$.
The random domain problem considered is to find the solution $u: \Omega \times \DDw \to \R$, such that
\begin{align*}
- \frac{d}{dx} \left( \big((x - 0.5)^2 + 1\big) \frac{d}{dx} u(\omega, x) \right) 
	& = 
1 
	&& \hspace*{-1.5cm}
x \in \Dw, \, \omega \in \Omega, 	
	\\
u(\omega, x)
	& =
0
	&& \hspace*{-1.5cm}
x \in \partial \Dw, \, \omega \in \Omega.
\end{align*}
There exists a unique solution to this random domain problem, which is analytically given by 
\begin{align*}
u(\omega, x) 
	 &=
\frac{
        \log\left(
        (5 \lvert Y(\omega) \rvert - 0.4)^2 + 1 \right) - \log(1.25)
}{
    2(\arctan(5 \lvert Y(\omega) \rvert - 0.4) + \arctan(0.5))
}
(\arctan(x - 0.5) + \arctan(0.5))
	- \frac{1}{2} \log\left(\frac{(x - 0.5)^2 + 1}{1.25} \right).
\end{align*}
The transformed diffusion field $\Ar : \Omega \times \Dr \to \R$ and transformed forcing term $\Fr: \Omega \times \Dr \to \R$ read
\begin{align*}
\Ar(\omega, x) 
	= \frac{1}{0.1 + 5 \lvert Y(\omega) \rvert}
	\big(
		(( 0.1 + 5 \lvert Y(\omega) \rvert ) x - 0.5)^2 + 1
	\big) , \quad
\Fr(\omega, x) 
	= 0.1 + 5 \lvert Y(\omega) \rvert.
\end{align*}
We choose the enforcer function $e(x) = x (1-x)$ on the deterministic reference domain and approximate the corresponding spectral coefficients by
\begin{align*}
{\scriptstyle \mathcal{U}}_{\theta}^{\operatorname{SG}}(x) 
	= x(1-x)\mathcal{N}_{\theta}^{\operatorname{SG}}(x), \quad
{\scriptstyle \mathcal{U}}_{\theta}^{\operatorname{SR}}(x) 
	= x(1-x)\mathcal{N}_{\theta}^{\operatorname{SR}}(x).
\end{align*}
We use the normalized Legendre polynomials evaluated in $Y \sim \mathcal{U}([-1,1])$ as a polynomial chaos basis of $L^2([-1,1]; \R)$ with respect to the uniform distribution. The experiments are carried out with maximal polynomial degrees $P = 0, \dots, 30$ for the stochastic Galerkin approximations. 
The spectral coefficients of the polynomial chaos expansion of the transformed matrix diffusion field and forcing term are analytically available. The corresponding Galerkin coupling operators $\mathbf{A}: \Dr \to \R^{(M + 1) \times (M + 1)}, \, \mathbf{B}: \Dr \to \R^{(M + 1) \times (M + 1)}$ are evaluated online for every training batch using the analytically derived spectral coefficients and corresponding derivatives.  
The resulting Galerkin systems yield coupled systems of dimension $M + 1 = P + 1$.
Both neural networks use the same fully connected feedforward architecture, consisting of six hidden layers, each consisting of 45 Swish activated neurons (see \citep{Ramachandran2017}) and a linear output layer consisting of $M + 1$ neurons. 
The \textit{S-GalerkinNet} is trained on a continuous stream of non-repeating Sobol points in batches of $1024$.
In this experiment, the \textit{S-RitzNet} optimization is very unstable and significantly worse than that of the \textit{S-GalerkinNet} when trained on a continuous stream or fixed grid of Sobol points for various batch sizes and learning rate schedules. 
To achieve a similar error performance with the \textit{S-RitzNet}, the neural network is trained on a fixed grid of $128$ Gauss-Legendre quadrature points, where the empirical risk is weighted with the corresponding quadrature weights.
Furthermore, in order to achieve a similar error performance to that of the \textit{S-GalerkinNet}, the optimization must resolve changes in the approximated Ritz energy functional that are too small to be reliably represented in single precision computations. Therefore the training of the \textit{S-RitzNet} is performed in double precision, while for the training of the \textit{S-GalerkinNet} single precision is sufficient.

The second order spatial derivatives to assemble the loss function of the \textit{S-GalerkinNet} are obtained via TensorFlow's forward-over-forward automatic differentiation, implemented using two nested forward-mode Jacobian-vector products. The \textit{S-RitzNet} only requires first order spatial derivatives, which are obtained by a single forward-mode Jacobian-vector product. This results in a reduction in GPU execution time for calculating a \textit{S-RitzNet} forward pass together with the spatial neural network derivatives, compared to the GPU execution time for the \textit{S-GalerkinNet}, by $24$ percent. The GPU execution time for evaluating the Galerkin coupling operators, is reduced by approximately $9$ percent for the \textit{S-RitzNet}.
The derivatives of both neural networks with respect to the trainable parameters are obtained via TensorFlow's reverse-mode automatic differentiation. 
There is no significant difference between the two training strategies in the GPU execution time required for the reverse-mode derivatives together with the optimizer step.
A forward evaluation together with the nested spatial Jacobian-vector products accounts for approximately $24$ percent of the GPU execution time of a \textit{S-GalerkinNet} training step.
The operations necessary to evaluate the Galerkin coupling operators and to assemble the corresponding loss functions and risk evaluations are computationally not demanding. For the \textit{S-GalerkinNet}, this step accounts for only $2$ percent of the GPU execution time, while the computation of the neural network derivatives with respect to the trainable parameters, together with the Adam optimizer step, accounts for the remaining $74$ percent. 
For the \textit{S-RitzNet} a forward evaluation together with the first order spatial derivative accounts for $19$ percent of the GPU execution time of one training step. The computational cost of evaluating the Galerkin coupling operator and assembling the loss function together with the quadrature rule to obtain the empirical risk is not enough to make a noticeable difference in the computational time in double precision. This part of the training step also accounts for approximately $2$ percent of the GPU execution time. The reverse-mode automatic differentiation together with the Adam optimizer step accounts for $79$ percent of the GPU time of one training step.
The GPU execution time for a complete \textit{S-RitzNet} training step is reduced by approximately $7$ percent compared to the GPU execution time of a complete \textit{S-GalerkinNet} training step.
All time profiles are obtained using the training times for $P = 30$, and the respective training configurations, including the different floating point precision.
\begin{figure}[!htbp]
\begin{minipage}{.5\textwidth}
  \begin{center}
  \includegraphics[width=0.99\textwidth]{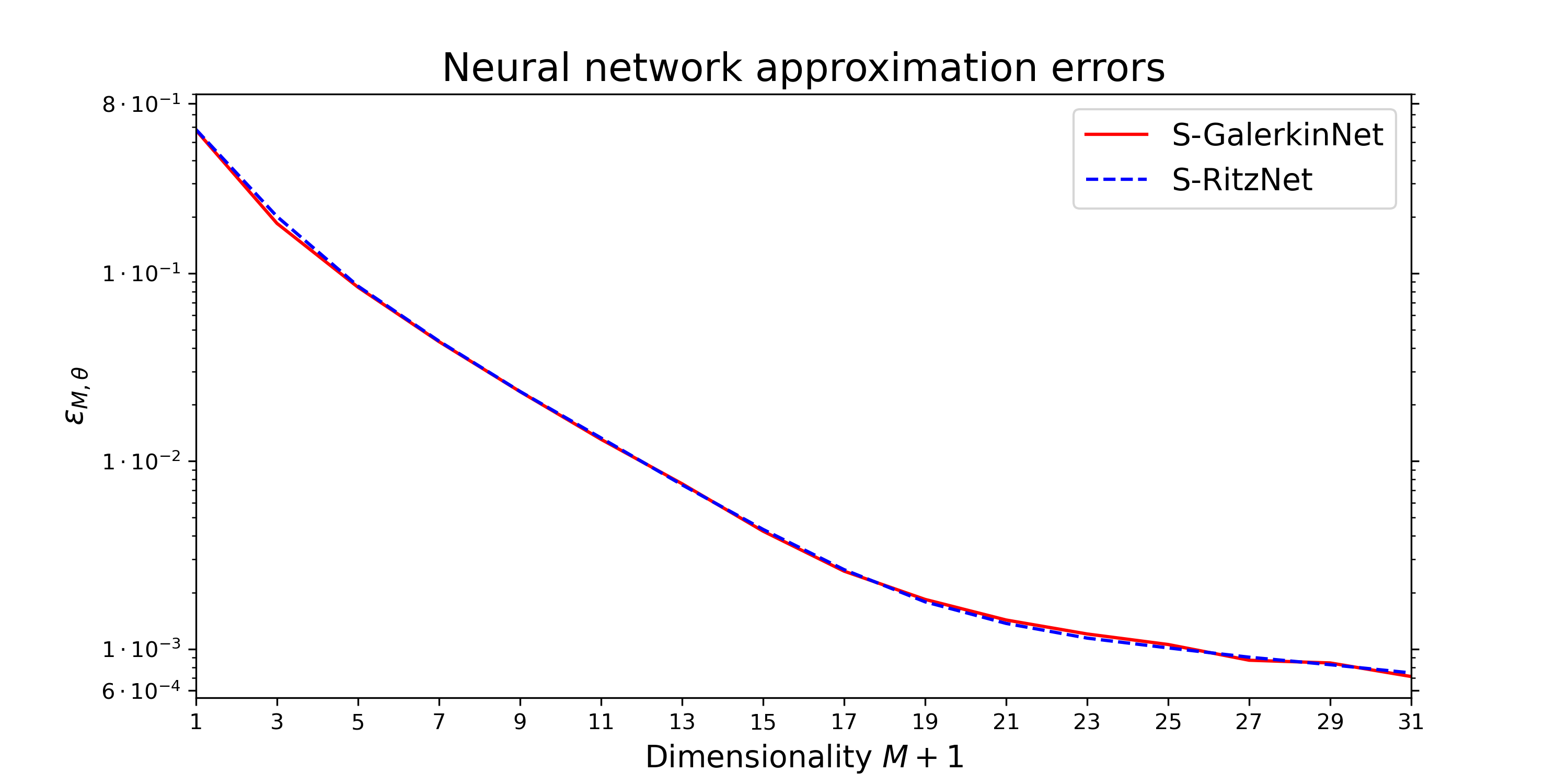} 
  \end{center}
\end{minipage}
\begin{minipage}{.5\textwidth}
  \begin{center}
  \includegraphics[width=0.99\textwidth]{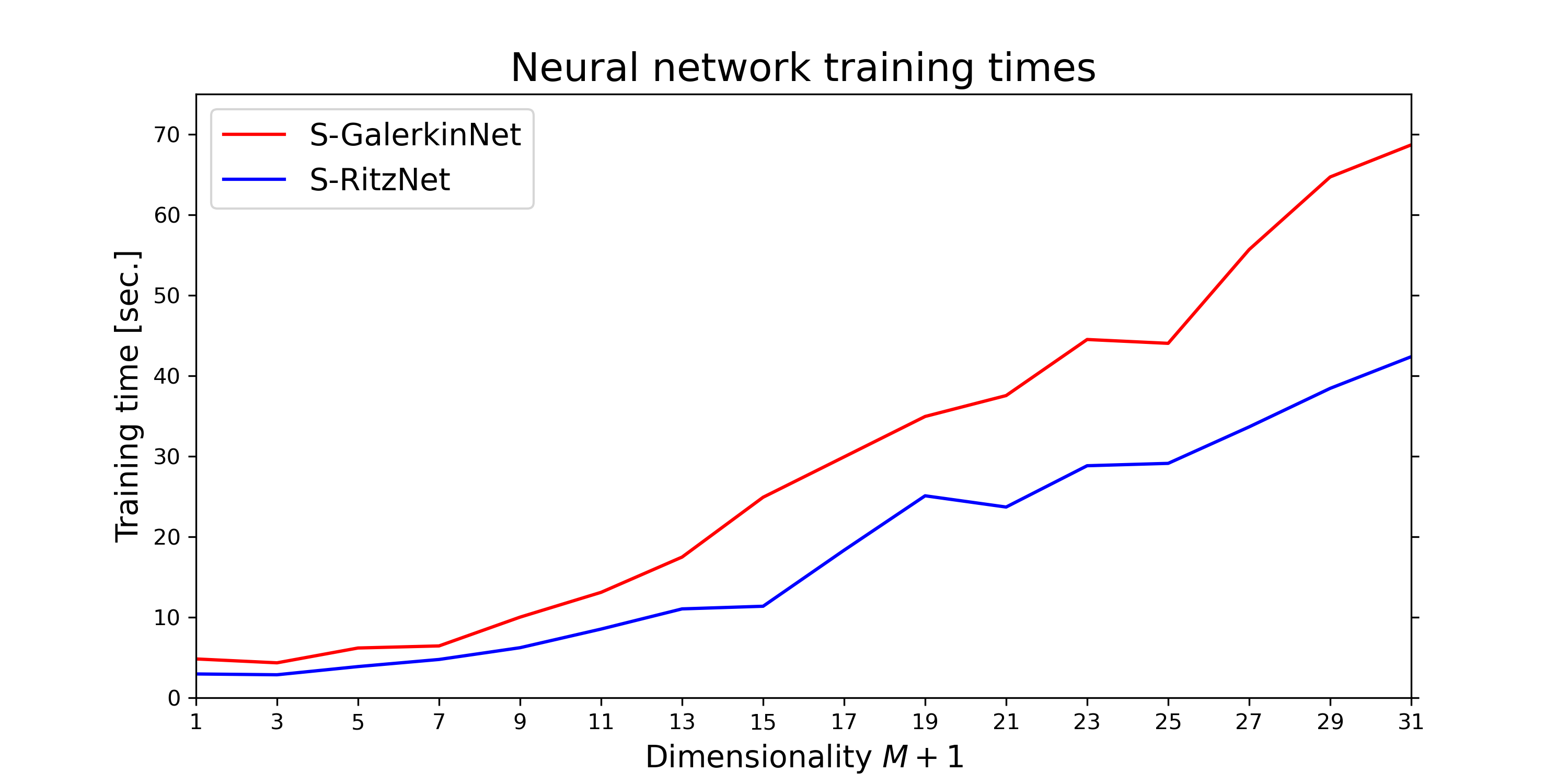}
  \end{center}
\end{minipage}
\caption{Performance and training times of \textit{S-GalerkinNet} and \textit{S-RitzNet}.}
\label{Figure:Experiment1_Data}
\end{figure}

In Figure \ref{Figure:Experiment1_Data}, we present the approximation errors as well as the training times of the neural networks.
The considered experiments are conducted using a single driving random variable, i.e.\ $N = 1$, and maximal total polynomial degrees $P \in \{0, 2, 4, \dots, 30\}$ for the generalized polynomial chaos expansions, resulting in coupled systems of dimensions $M + 1 \in \{1, 3, 5, \dots, 31\}$. Since every odd spectral coefficient of the transformed diffusion coefficient and forcing term polynomial chaos representation is zero, no meaningful information enters the Galerkin system by increasing the maximal total polynomial degree from $2n$ to $2n + 1$, $n \in \N_0$. Therefore we only consider even increments of the polynomial degree $P$.
The reported timings include the whole online phase, including TensorFlow's graph tracing. 
All presented $\epsilon_{M, \theta}$-errors are computed using Monte Carlo integration with $2 \cdot 10^4$ independent sample paths of the analytically available reference solution over the stochastic domain, and via the trapezoidal quadrature rule on $5001$ spatial points that are not used during the training procedure. 
For $P = 30$, both deep learning approaches achieve a relative $\epsilon_{30, \theta}$-error below $0.075\%$.
Figure~\ref{Figure:Experiment4_random_line_rd_ref_sol} shows $30$ paths of deep learning stochastic Galerkin approximations, generated by the \textit{S-RitzNet}
\begin{align*}
u^{(30)}_{\theta_2^*}(\omega, x) &=
\sum_{k = 0}^{30} {\scriptstyle \mathcal{U}}_{k;\theta_2^*}^{\operatorname{SR}}(x) \operatorname{Le}_k(Y(\omega)), \quad x \in \DDr,
	\\
u^{(30)}_{\theta_2^*}(\omega, x) &=
\sum_{k = 0}^{30} {\scriptstyle \mathcal{U}}_{k;\theta_2^*}^{\operatorname{SR}}(\mathcal{T}^{-1}(\omega, x)) \operatorname{Le}_k(Y(\omega)), \quad x \in \DDw
\end{align*}
for $P = 30$ on the reference domain, as well as the same paths of corresponding solution approximations mapped back to the random domain. 
Figure~\ref{Figure:Experiment4_random_line_sol_analytical_plots} shows two paths of the deep learning stochastic Galerkin approximations of \textit{S-GalerkinNet} and \textit{S-RitzNet}, given by
\begin{align*}
u^{(M)}_{\theta_1^*}(\omega, x) &=
\sum_{k = 0}^{M} {\scriptstyle \mathcal{U}}_{k;\theta_1^*}^{\operatorname{SG}}(\mathcal{T}^{-1}(\omega, x)) \operatorname{Le}_k(Y(\omega)),
	\\
u^{(M)}_{\theta_2^*}(\omega, x) &=
\sum_{k = 0}^{M} {\scriptstyle \mathcal{U}}_{k;\theta_2^*}^{\operatorname{SR}}(\mathcal{T}^{-1}(\omega, x)) \operatorname{Le}_k(Y(\omega)),
\end{align*}
on the random domain together with their spatial derivatives and the corresponding analytical solution paths and derivatives. 
The first path depicts a realization $Y(\omega_1) \approx 0.06$, and maximal total polynomial degrees $P \in \{0, 12, 30\}$, resulting in system dimensions $M + 1 \in \{1, 13, 31\}$. 
We see that the convergence of the paths to the analytical solution requires a large total polynomial degree. The second path depicts a realization $Y(\omega_2) \approx 0.99$, and maximal total polynomial degrees $P \in \{0, 4, 12\}$, resulting in the system dimensions $M + 1 \in \{1, 5, 13\}$. We see that a maximal total polynomial degree of $P = 12$ already suffices for the corresponding graphs and their derivatives to agree sufficiently well.
\begin{figure}[!htbp]
\begin{minipage}{.5\textwidth}
  \begin{center}
  \includegraphics[width=0.99\textwidth]{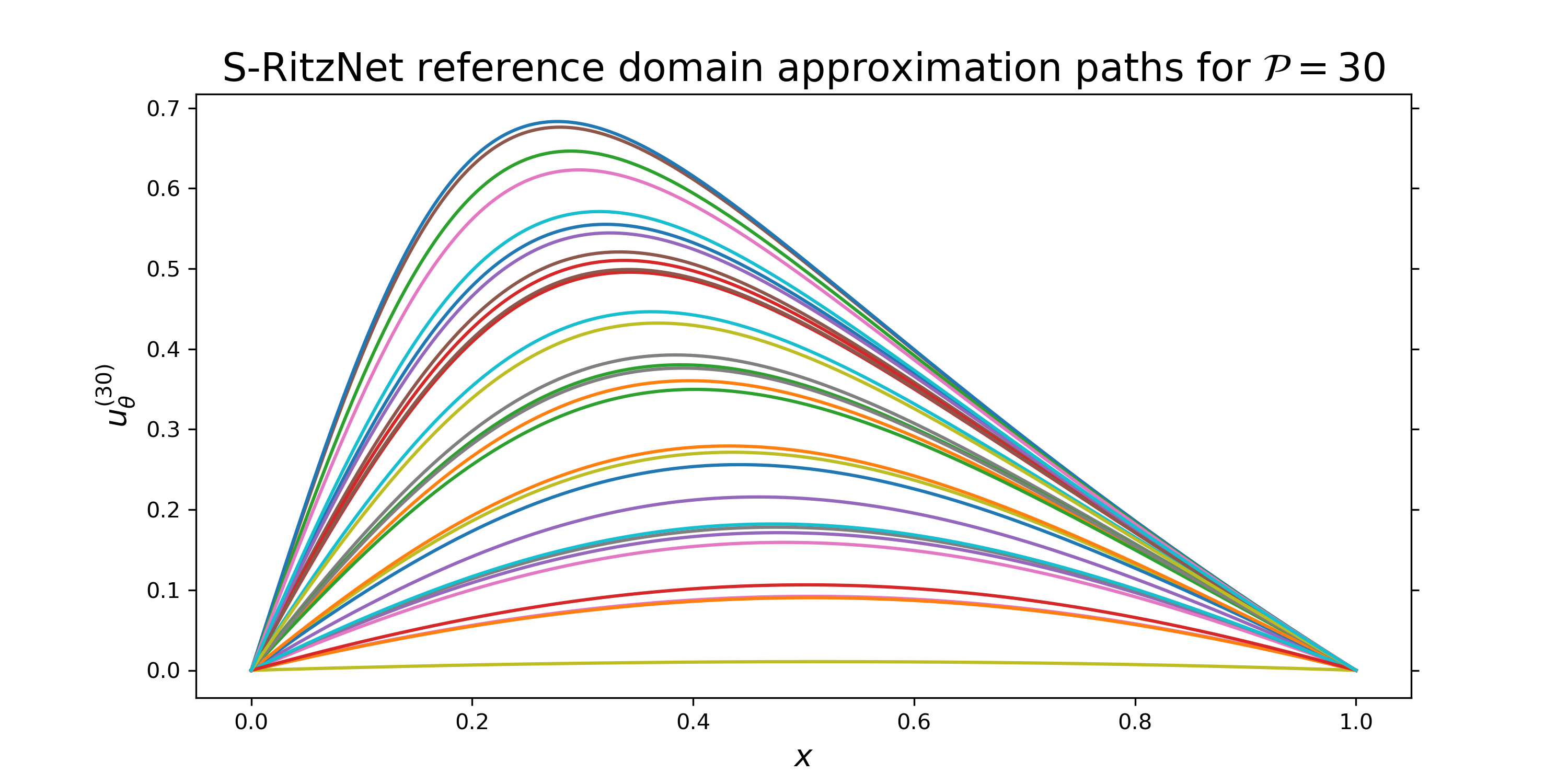} 
  \end{center}
\end{minipage}
\begin{minipage}{.5\textwidth}
  \begin{center}
  \includegraphics[width=0.99\textwidth]{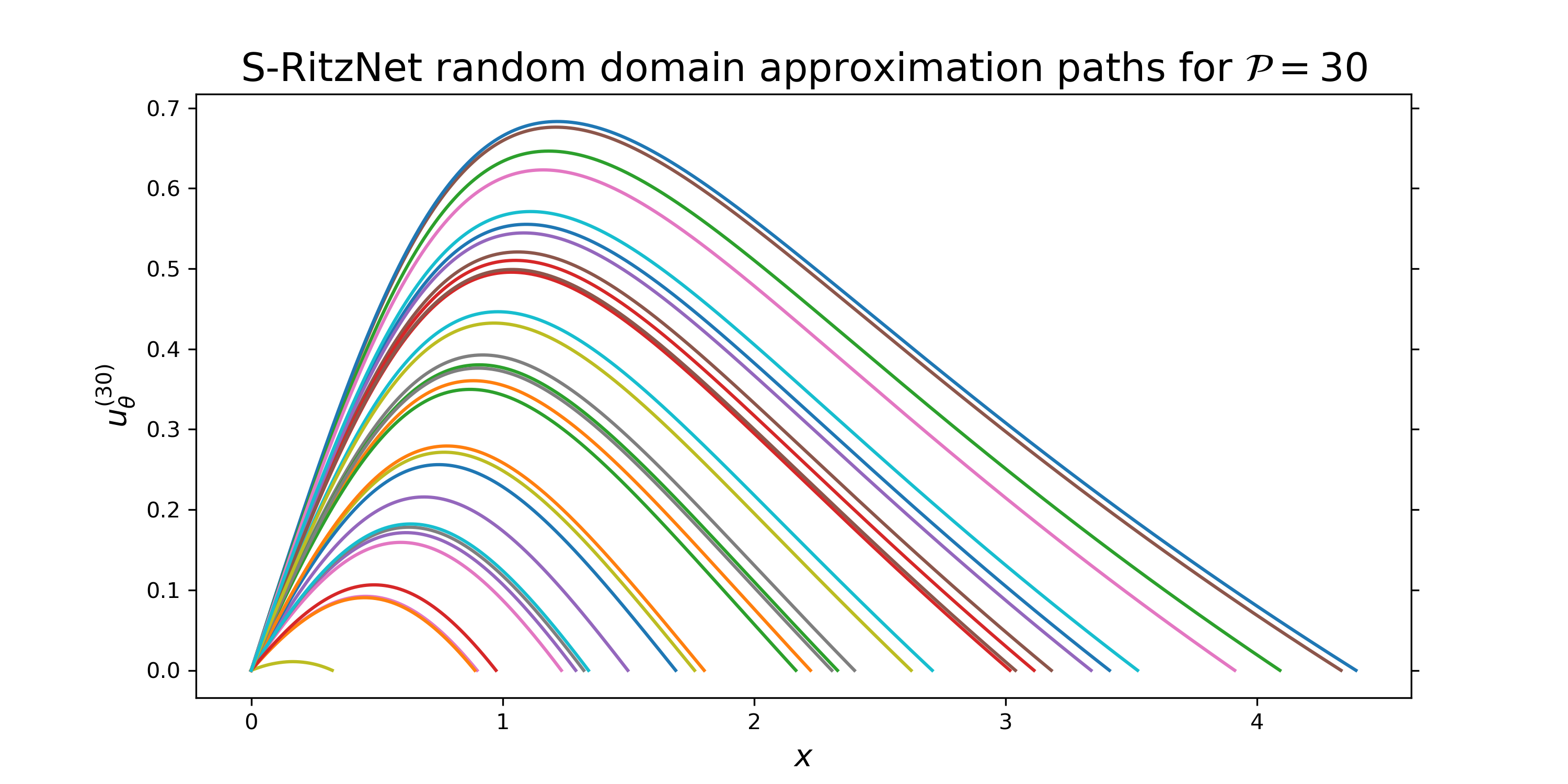}
  \end{center}
\end{minipage}
\caption{\textit{S-RitzNet} solution approximation on the reference and random domain.}
\label{Figure:Experiment4_random_line_rd_ref_sol}
\end{figure}
\begin{figure}[!htbp]
	\centering
  \includegraphics[width=0.99\textwidth]{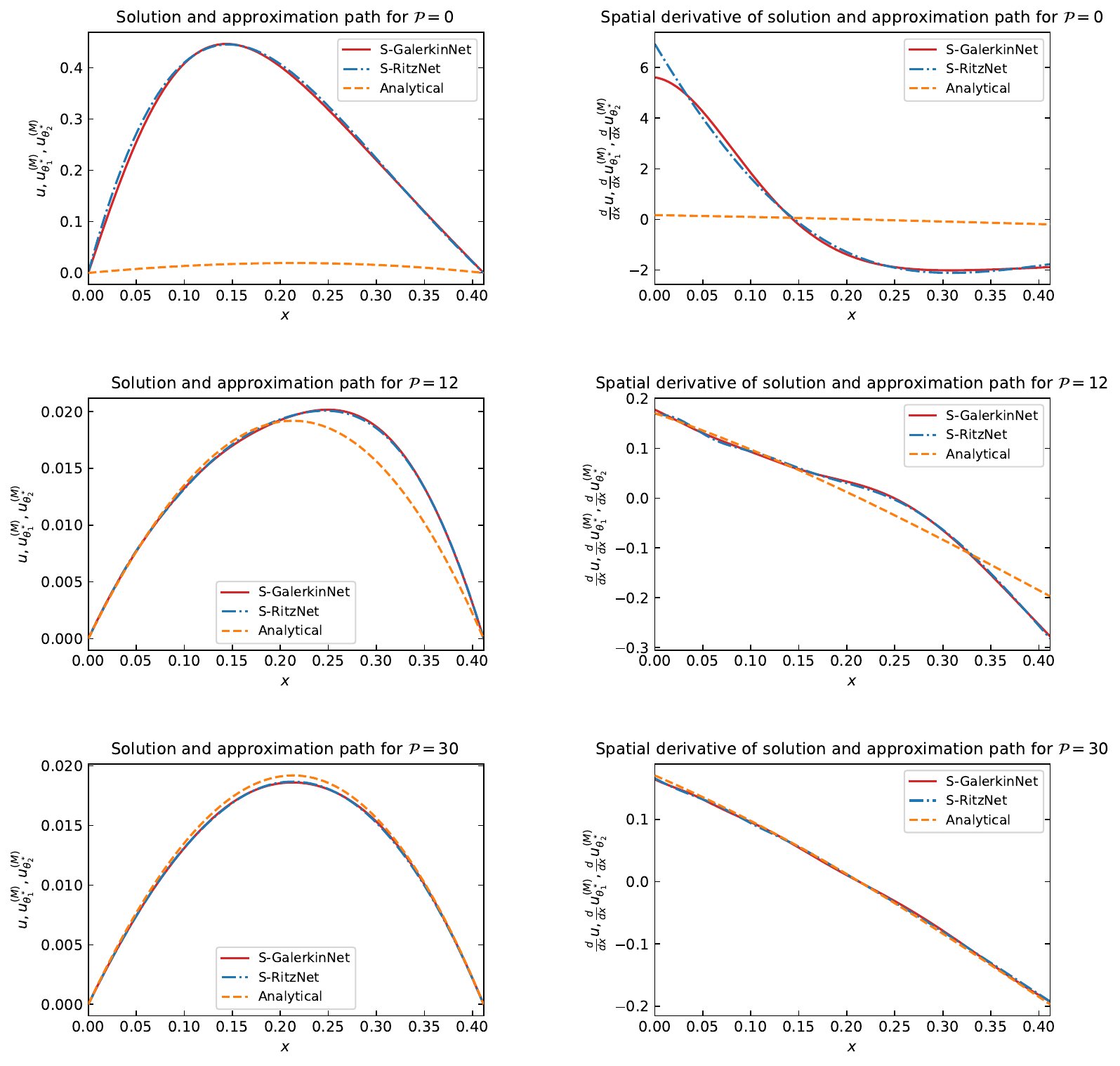} 
  \captionsetup{labelformat=empty}
  \caption{\,}
\label{Figure:Experiment4_random_line_sol_analytical_plots}
\end{figure}
 \clearpage
  \begin{figure}[H]
  \ContinuedFloat
  \centering
  \includegraphics[width=0.99\textwidth]{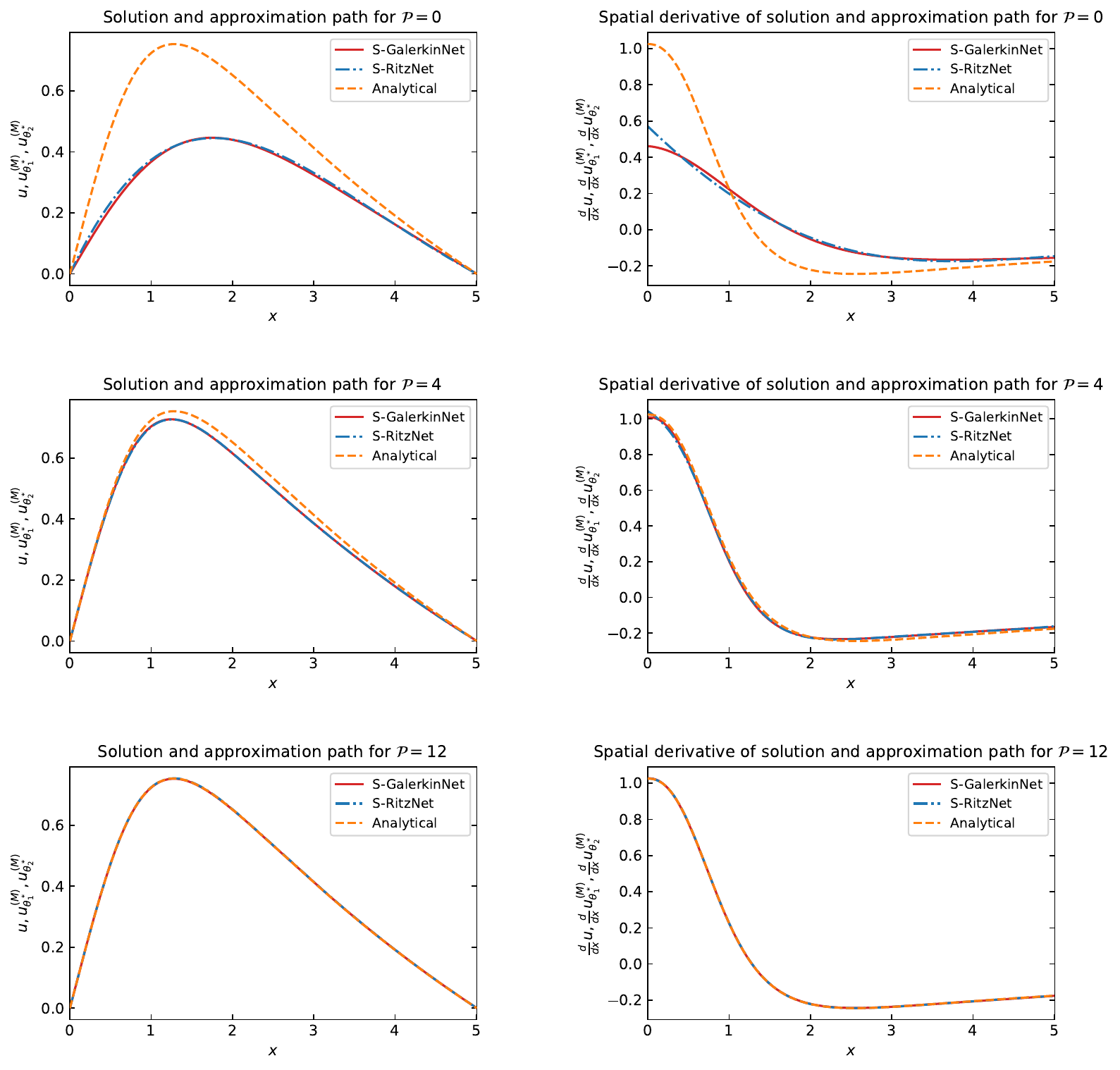}
  \caption{Analytical solution paths and their spatial derivatives together with the S-GalerkinNet and S-RitzNet approximations.}
\end{figure}

As a second example, we consider the elliptic equation in two spatial dimensions posed on a randomly deformed annulus. 
Let the reference domain be given by $\DDr = \{(r, \varphi) \in \R^2 \, : \, r \in [0.5, 1], \, \varphi \in [0, 2\pi) \}$, parameterizing the reference annulus in polar coordinates ($(x_1, x_2) = (r \cos(\varphi), r \sin(\varphi))^T$).
The random domain mapping $\mathcal{T}: \Omega \times [0.5, 1] \times [0, 2 \pi) \to \R^2$, is defined by
\begin{align*}
\mathcal{T}(\omega, r, \varphi) := 
	\begin{pmatrix}
		r \cos(\varphi) \\
		r \sin(\varphi)
	\end{pmatrix}
+ \sum_{k = 1}^{30}
\frac{1}{2(1+k^2)}
\cos(k \varphi)
	\begin{pmatrix}
		r \cos(\varphi) \\
		r \sin(\varphi)
	\end{pmatrix}
Y_k(\omega),
\end{align*}
where $(Y_k)_{k = 1}^{30}$ are independent random variables with $Y_k \sim \mathcal{U}([-1,1])$ for $k = 1, \dots, 30$. This representation highlights that the random domain mapping is given as a KL-type expansion, but it can also be written as
\begin{align*}
\mathcal{T}(\omega, r, \varphi) = r \alpha(\omega, \varphi)
	\begin{pmatrix}
		\cos(\varphi) \\
		\sin(\varphi)
	\end{pmatrix}, 
\quad
\alpha(\omega, \varphi) = 1 + \sum_{k = 1}^{30} \frac{1}{2(1 + k^2)} \cos(k \varphi) Y_k(\omega).
\end{align*}
This representation shows that the deformation of the reference annulus is purely radial and only angle-dependent. 
The Jacobian of the random domain mapping is analytically given by
\begin{align*}
J_{\mathcal{T}}(\omega, r, \varphi) = 
	\begin{pmatrix}
		\alpha(\omega, \varphi) \cos(\varphi) 
		& r \alpha'(\omega, \varphi) \cos(\varphi) - r \alpha(\omega, \varphi) \sin(\varphi) 
			\\[4pt]
		\alpha(\omega, \varphi) \sin(\varphi) 
		& r \alpha'(\omega, \varphi) \sin(\varphi) + r \alpha(\omega, \varphi) \cos(\varphi)			
	\end{pmatrix},
\end{align*}
where $\alpha'(\omega, \varphi)$ refers to the derivative with respect to the angle $\varphi$.
The functional determinant $\operatorname{det}(J_{\mathcal{T}}(\omega, r, \varphi)) = r \alpha^2(\omega, \varphi)$ is uniformly bounded away from zero for every $(\omega, r, \varphi) \in \Omega \times [0.5, 1] \times [0, 2\pi)$. Furthermore, the random domain mapping satisfies Assumption~\ref{Assumption:AdditionalAssmStrongResidual}.

We consider the random domain problem of finding a solution $u: \Omega \times \DDw \to \R$, such that
\begin{align*}
- \nabla \cdot \left( \sqrt{x_1^2 + x_2^2} \, \nabla u(\omega, (x_1, x_2)) \right) 
	& = 
1 
	&& \hspace*{-3cm}
(x_1, x_2) \in \DD(\omega), \, \omega \in \Omega, 	
	\\
u(\omega, (x_1, x_2))
	& =
0
	&& \hspace*{-3cm}
(x_1, x_2) \in \partial \DD (\omega), \, \omega \in \Omega.
\end{align*} 
The transformed matrix diffusion field $\Ar: \Omega \times (0.5, 1) \times [0 , 2 \pi) \to \mathcal{M}_2^2$ and the transformed forcing term $\Fr : \Omega \times (0.5, 1) \times [0, 2\pi) \to \R$ are given in closed form via
\begin{align*}
\Ar(\omega, r, \varphi) = 
	\begin{pmatrix}
		r^2 \left(\alpha(\omega, \varphi) + \frac{(\alpha'(\omega, \varphi))^2}{\alpha(\omega, \varphi)} \right)
			& - r \alpha'(\omega, \varphi)
			\\[7pt]
			- r \alpha'(\omega, \varphi)
			& \alpha(\omega, \varphi)			
	\end{pmatrix},
\quad
\Fr(\omega, r, \varphi) = r \alpha^2(\omega, \varphi).
\end{align*}
The corresponding transformed elliptic random PDE reads: find a solution $u:\Omega \times [0.5, 1] \times [0, 2\pi) \to \R$, such that
\begin{align*}
- \nabla \cdot \left( \Ar (\omega, r, \varphi) \, \nabla u(\omega, r, \varphi) \right) 
	& = 
\Fr (\omega, r, \varphi)
	&& \hspace*{-1cm}
(r, \varphi) \in (0.5, 1) \times [0, 2\pi), \, \omega \in \Omega, 	
	\\
u(\omega, 0.5, \varphi) = u(\omega, 1, \varphi)
	& =
0
	&& \hspace*{-1cm}
\varphi \in [0, 2\pi), \, \omega \in \Omega.
\end{align*}
In the presented experiments, we keep the maximal total polynomial degree of the polynomial chaos expansions fixed $P = 2$, and increase the stochastic dimension in form of considered terms in the Karhunen-Loève-type expansion $N = 1, \dots, 8$.
Since the driving random variables are uniformly distributed, we choose the normalized multivariate tensor product Legendre polynomials as generalized polynomial chaos basis. 
The corresponding dimensions of the coupled Galerkin systems $M + 1 = \frac{1}{2} (N + 2) (N + 1)$ are determined by the number $N$ of terms used in the Karhunen-Loève-type expansion of the random domain mapping. 
More specifically, the truncated transformed matrix diffusion field and forcing term
\begin{align*}
A_{\text{ref}; N}(\omega, r, \varphi) =
\begin{pmatrix}
		r^2 \left(\alpha_N(\omega, \varphi) + \frac{(\alpha_N'(\omega, \varphi))^2}{\alpha_N(\omega, \varphi)} \right)
			& - r \alpha_N'(\omega, \varphi)
			\\[7pt]
			- r \alpha_N'(\omega, \varphi)
			& \alpha_N(\omega, \varphi)			
	\end{pmatrix},
\quad
f_{\text{ref}; N}(\omega, r, \varphi) = r \alpha_N^2(\omega, \varphi)
\end{align*} 
are expanded into their respective polynomial chaos expansions, where
\begin{align*}
\alpha_N(\omega, \varphi)	 = 1 + \sum_{k = 1}^{N} \frac{1}{2(1 + k^2)} \cos(k \varphi) Y_k(\omega). 
\end{align*}
The stochastic integrals to obtain the polynomial chaos expansions are approximated using quasi Monte Carlo integration once in the offline phase and evaluated on a fixed grid of spatial quadrature points.
Homogeneous Dirichlet boundary conditions are enforced via the function $e(r, \varphi) = 4(r - 1/2)(1 - r)$, leading to the neural network approximations
\begin{align*}
{\scriptstyle \mathcal{U}}_{\theta}^{\operatorname{SG}}(r,\varphi) 
	= 4(r - 1/2)(1 - r)\mathcal{N}_{\theta}^{\operatorname{SG}}(r, \varphi), \quad
{\scriptstyle \mathcal{U}}_{\theta}^{\operatorname{SR}}(r, \varphi) 
	= 4(r - 1/2)(1 - r)\mathcal{N}_{\theta}^{\operatorname{SR}}(r, \varphi).
\end{align*}
\begin{figure}
\begin{minipage}{.5\textwidth}
  \begin{center}
  \includegraphics[width=0.99\textwidth]{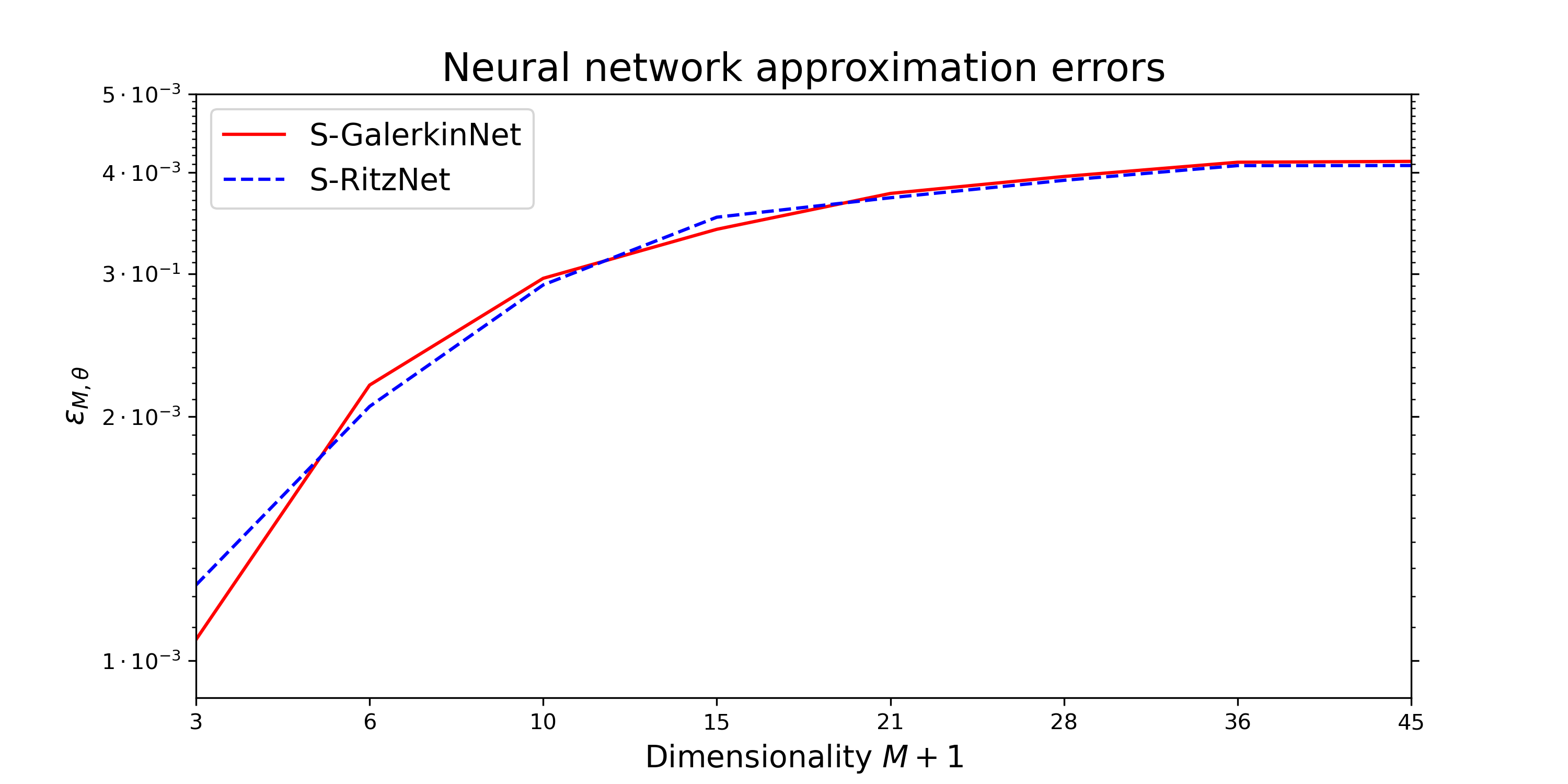} 
  \end{center}
\end{minipage}
\begin{minipage}{.5\textwidth}
  \begin{center}
  \includegraphics[width=0.99\textwidth]{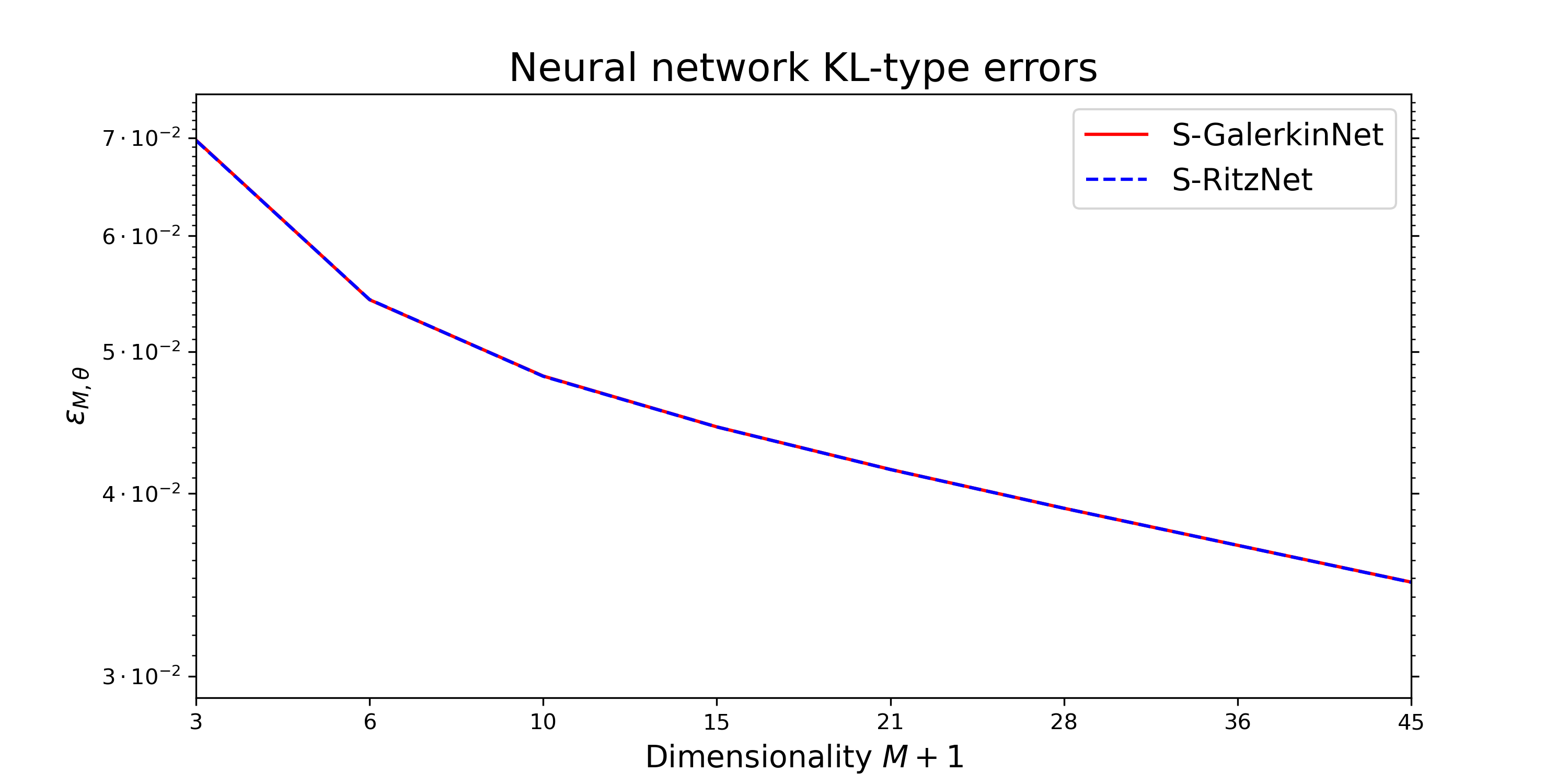}
  \end{center}
\end{minipage}

\begin{center}
\includegraphics[width=0.49\textwidth]{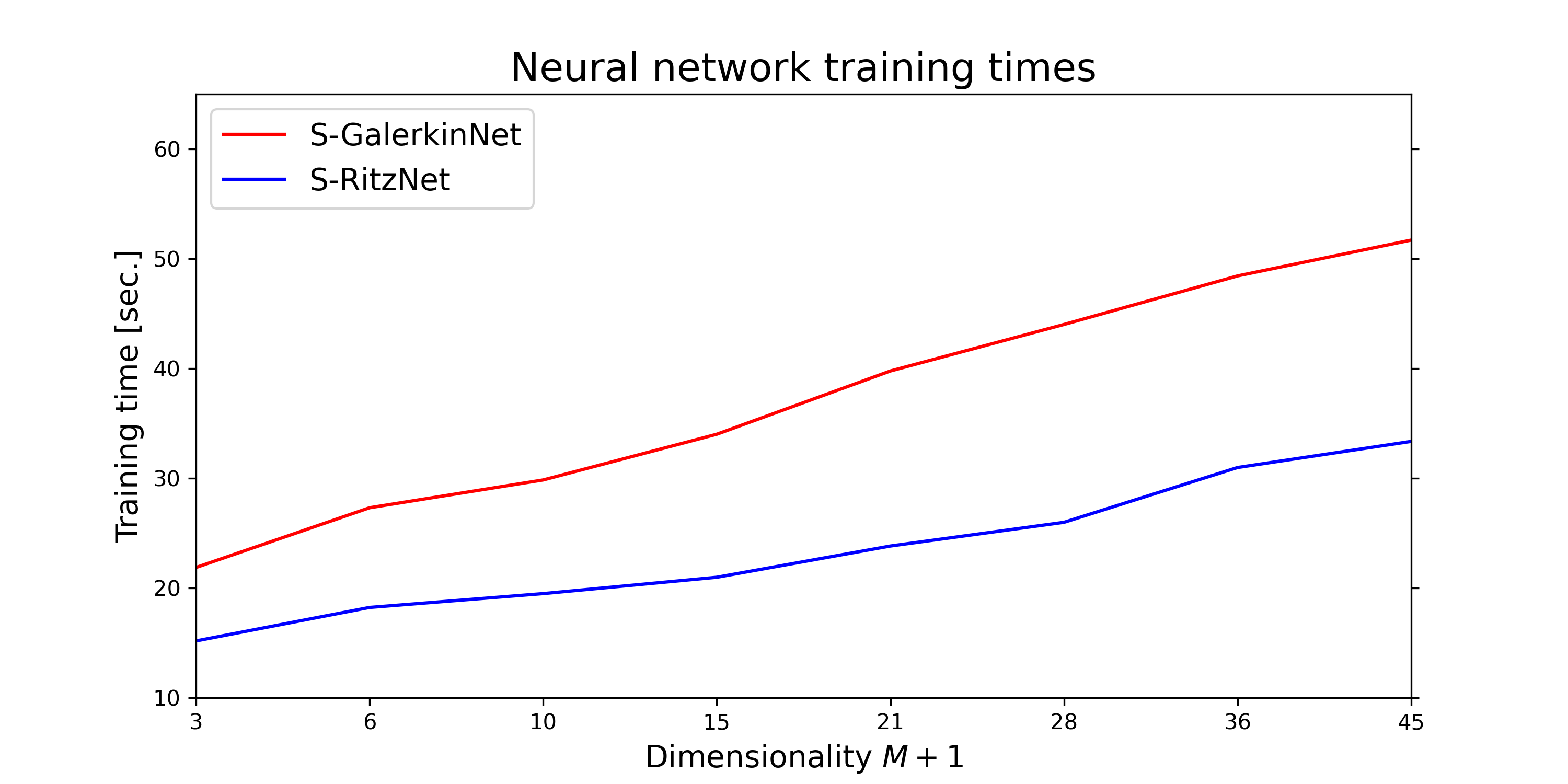}
\end{center}

\caption{Performance and training times of \textit{S-GalerkinNet} and \textit{S-RitzNet}.}
\label{Figure:Experiment2_Data}
\end{figure}
Both neural networks are built with the same architecture. 
To explicitly encode the periodic structure of the angular coordinate $\varphi$, the input variables $(r, \varphi) \in \DDr$ are mapped to $(r, \sin(\varphi), \cos(\varphi))$.
The transformed spatial inputs are first processed by a fully connected layer with $90$ sinusoidal activated neurons. This SIREN (see \citep{Sitzmann2020}) inspired approach enhances the high-frequency expressiveness of the neural networks, mitigating the spectral bias (see \citep{Rahaman2019}).
Next follows a block of six fully connected hidden layers with $90$ neurons, activated by the smooth Swish activation function.
The linearly activated output layer returns a $(M + 1)$-dimensional vector, where the $i$-th component approximates the $i$-th spectral coefficient of the stochastic Galerkin solution. 
The neural networks are trained on a fixed grid of tensor product quadrature points. We use the Gauss-Legendre quadrature rule, on $24$ quadrature points shifted to the interval $[0.5, 1]$, for the radial variable $r \in [0.5, 1]$, and the trapezoidal quadrature rule, on $48$ quadrature points, for the angular variable $\varphi \in [0, 2\pi)$. The corresponding empirical risk calculations are adjusted to include the quadrature weights.

The loss function of the \textit{S-GalerkinNet} requires the Hessian of the neural network output neurons with respect to the spatial inputs. The Hessian is obtained very efficiently via TensorFlow's forward-over-forward mode automatic differentiation, implemented as two batched nested Jacobian-vector products. For this, we stack three identical copies of the input spatial points and stack them into a batched vector and assign the batched tangent directions to obtain the spatial derivatives
\begin{align*}
\frac{\partial}{\partial r} \frac{\partial}{\partial r} {\scriptstyle \mathcal{U}}_{\theta}^{\operatorname{SG}}((r, \varphi)), \quad
\frac{\partial}{\partial r} \frac{\partial}{\partial \varphi} {\scriptstyle \mathcal{U}}_{\theta}^{\operatorname{SG}}((r, \varphi)), \quad
\frac{\partial}{\partial \varphi} \frac{\partial}{\partial \varphi} {\scriptstyle \mathcal{U}}_{\theta}^{\operatorname{SG}}((r, \varphi)), \quad
\end{align*} 
simultaneously in a single batched forward-over-forward evaluation. The Hessian is then assembled utilizing its symmetry for smooth neural networks.
For the \textit{S-RitzNet} only the gradients of the neural network output neurons are necessary, which are obtained by a single batched forward-mode Jacobian-vector product. For this only two identical copies of the input spatial points are stacked into a batched input vector with corresponding tangent directions, to obtain the gradient in a single batched forward-pass. 
Due to this simplified computational architecture, the GPU execution time for a forward pass together with the computation of the necessary derivatives for the \textit{S-RitzNet} is reduced by $79$ percent compared to the \textit{S-GalerkinNet}. The neural network derivatives with respect to the trainable parameters are obtained using TensorFlow's reverse-mode automatic differentiation. The backpropagation through the simplified computational architecture of the \textit{S-RitzNet} results in a $64$ percent reduced GPU execution time to obtain the network parameter derivatives together with the Adam optimizer step compared to the \textit{S-GalerkinNet}.
For the \textit{S-GalerkinNet} a forward pass together with the nested batched Jacobian-vector products to obtain the Hessian of the output neurons account for approximately $20$ percent of the GPU execution time for one training step. Since the Galerkin coupling operators are assembled and evaluated once in the offline phase on the fixed grid of tensor product quadrature points, the assembly of the loss function and computation of the risk has a negligible share of the GPU execution time of approximately $1$ percent of a training step. The main computational effort, accounting for approximately $79$ percent of the GPU execution time, arises from the reverse-mode automatic differentiation to obtain the neural network parameter derivatives together with the Adam optimizer step. 
For the \textit{S-RitzNet}, we see a similar breakdown of the GPU execution times, where the loss and risk assembly take up a larger portion, due to the reduced GPU execution times of the forward and backward propagations. The forward pass together with the single batched Jacobian-vector product accounts for approximately $12$ percent of the GPU execution time of a training step. The assembly of the loss and the risk evaluation account for approximately $6$ percent, while the reverse-mode automatic differentiation together with the Adam optimizer step account for the remaining $82$ percent of the GPU execution time. The GPU execution time of a complete training step of the \textit{S-RitzNet} is reduced by $66$ percent compared to the \textit{S-GalerkinNet}, although fewer optimizer steps are required for the \textit{S-GalerkinNet} to reach the error levels shown.
All the time profiles and percentages are obtained for the training of the neural networks using $N = 8$ terms in the Karhunen-Loève-type expansion.

In Figure~\ref{Figure:Experiment2_Data}, we present the error performance and corresponding training times of the \textit{S-GalerkinNet} and the \textit{S-RitzNet}. All the reported timings present the whole online phase, including TensorFlow's graph tracing. The presented $\epsilon_{M, \theta}$-errors are computed using $10^4$ different solution paths on random domain realizations that are generated pathwise via the finite element method with the Python library FEniCS \cite{Alnaes2015, Logg2012}. The stochastic integrals are approximated using Monte Carlo integration with these solution paths, and the spatial integrals are approximated numerically, for every solution path, on a grid of $22500$ spatial points that are not used during the training. The depicted neural network approximation errors are computed with FEniCS reference solutions that are generated using the same Karhunen-Loève-type truncation index $N = 1, \dots, 8$ for the random domain mapping that is used for the corresponding neural networks. The presented Karhunen-Loève-type errors are computed with FEniCS reference solutions that use the full finite Karhunen-Loève-type expansion of the random domain mapping, consisting of $N = 30$ terms, to additionally include the error introduced by truncating the Karhunen-Loève-type expansion. 
The relative $\epsilon_{M, \theta}$ neural network approximation errors range from just above $0.1\%$ for $N = 1$, resulting in a system dimensionality $M + 1 = 3$ to just above $0.4\%$ for $N = 8$, resulting in a system dimensionality $M + 1 = 45$, for both deep learning approaches. The Karhunen-Loève-type errors, computed against reference solutions with the full $N = 30$ terms in the Karhunen-Loève-type expansion of the random domain mapping range from $7\%$ for $N = 1$ to just under $3.5\%$ for $N = 8$.

Figure~\ref{Figure:Experiment5_random_annulus_solution_paths} shows six different paths of the
deep learning stochastic Galerkin approximations of the \textit{S-GalerkinNet} and the \textit{S-RitzNet}, given by
\begin{align*}
u^{(44)}_{\theta_1^*}(\omega, (r, \varphi)) &=
\sum_{k = 0}^{44} {\scriptstyle \mathcal{U}}_{k;\theta_1^*}^{\operatorname{SG}}(\mathcal{T}^{-1}(\omega, (r, \varphi))) \operatorname{Le}_k(Y(\omega)),
	\\
u^{(44)}_{\theta_2^*}(\omega, (r, \varphi)) &=
\sum_{k = 0}^{44} {\scriptstyle \mathcal{U}}_{k;\theta_2^*}^{\operatorname{SR}}(\mathcal{T}^{-1}(\omega, (r, \varphi))) \operatorname{Le}_k(Y(\omega)),
\end{align*}
on the random domain, with $Y = (Y_1, \dots, Y_{8})$. Three paths are approximated by the \textit{S-GalerkinNet} and three paths are approximated by the \textit{S-RitzNet} with $N = 8$, and corresponding system dimension $M + 1 = 45$.

Figure~\ref{Figure:Experiment5_random_annulus_RDG_path_and_error} shows one realization of the \textit{S-GalerkinNet} solution approximation on the random domain together with its partial radial and angular derivatives and the corresponding pointwise absolute errors against the FEniCS solution path using the same truncation index $N = 8$. In Figure~\ref{Figure:Experiment5_random_annulus_RDR_path_and_error}, the same is depicted for a different random domain realization for the \textit{S-RitzNet}.

\begin{figure}[!htbp]
\centering
\includegraphics[width=0.99\textwidth]{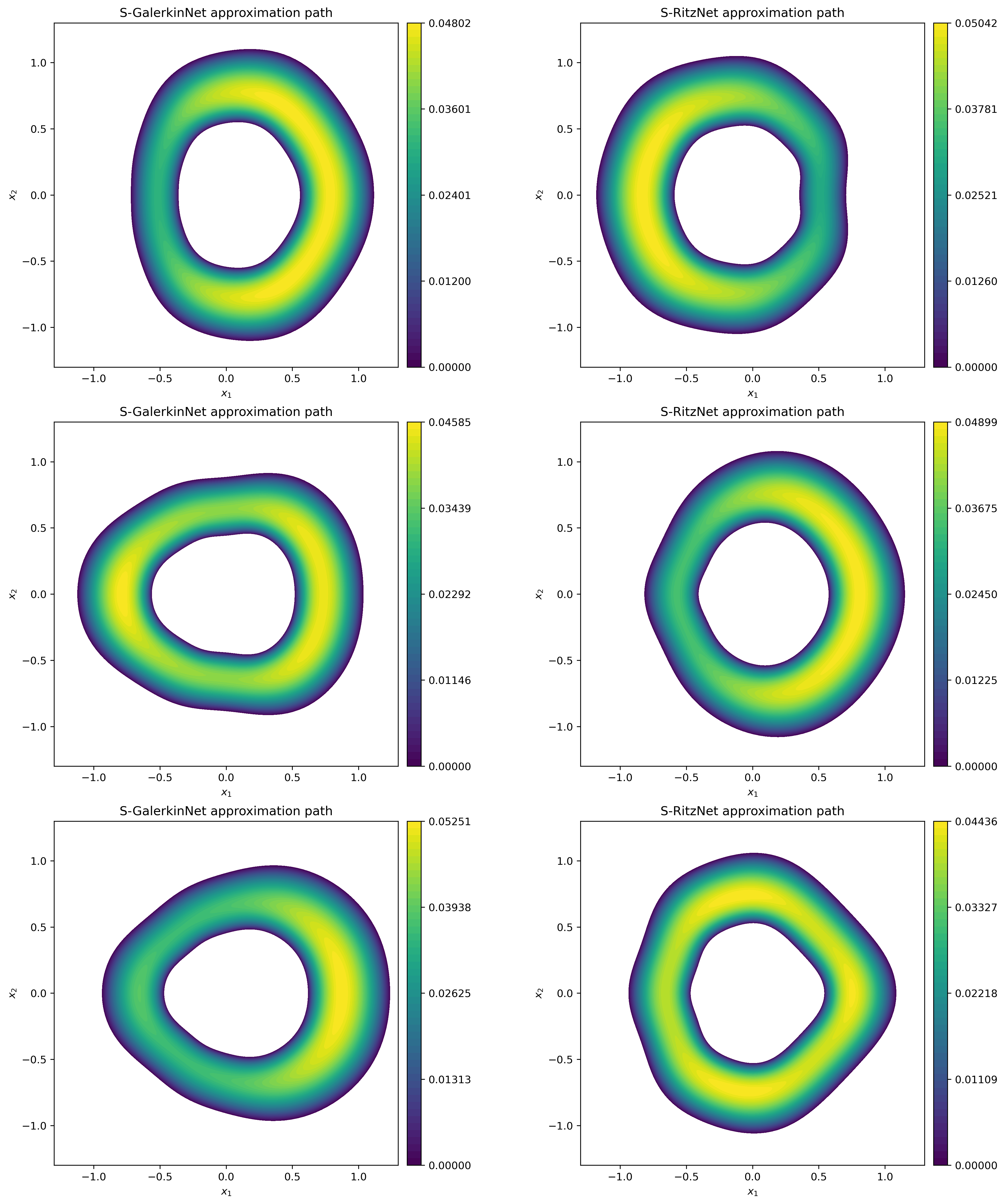} 
\caption{Different random domain solution approximations by the \textit{S-GalerkinNet} and \textit{S-RitzNet}.}
\label{Figure:Experiment5_random_annulus_solution_paths}
\end{figure}

\begin{figure}[!htbp]
\centering
\includegraphics[width=0.99\textwidth]{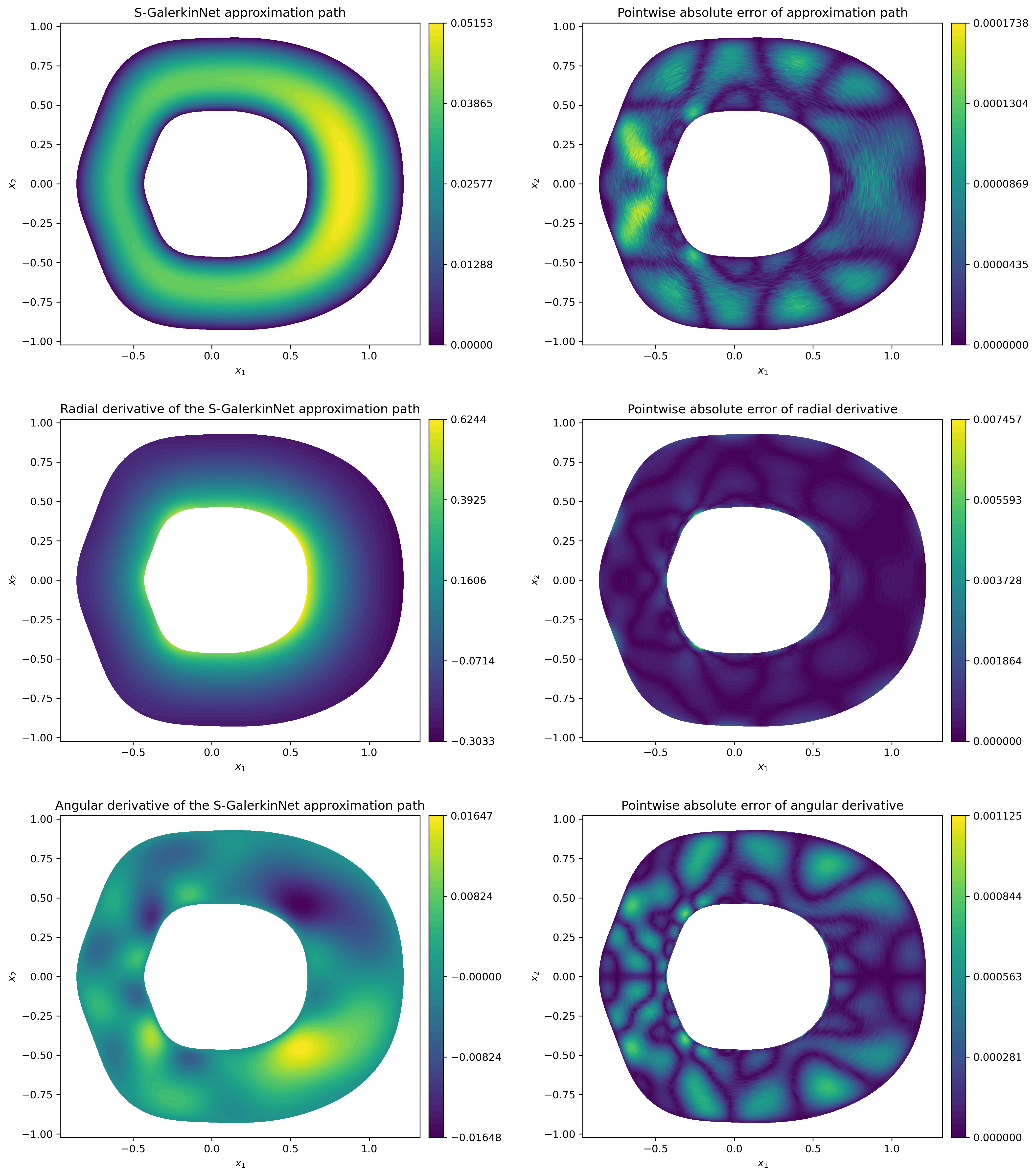} 
\caption{\textit{S-GalerkinNet} approximation of one solution path with pointwise absolute errors.}
\label{Figure:Experiment5_random_annulus_RDG_path_and_error}
\end{figure}

\begin{figure}[!htbp]
\centering
\includegraphics[width=0.99\textwidth]{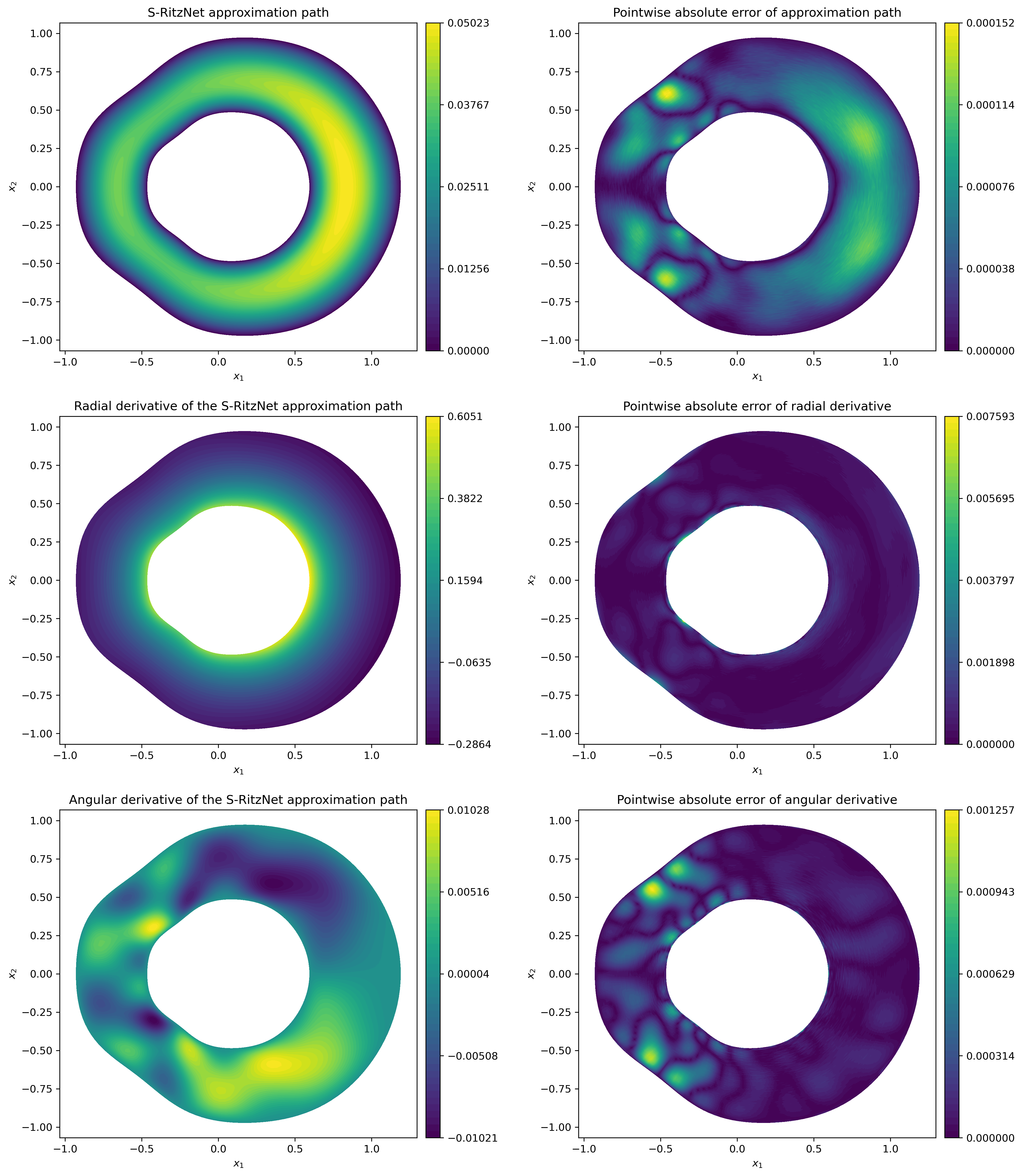} 
\caption{\textit{S-RitzNet} approximation of one solution path with pointwise absolute errors.}
\label{Figure:Experiment5_random_annulus_RDR_path_and_error}
\end{figure}

\newpage
\bibliography{references_dlrd}

\end{document}